\documentclass[final]{siamltex}
\usepackage{amsmath, amsfonts, amssymb, epsfig}
\usepackage{color}
\usepackage{url}

\def\0{{\bf 0}}
\def\1{{\bf 1}}

\def\R{\mathbb{R}}

\def\1{{\bf 1}}

\def\H{{\cal H}}

\newcommand{\hh}{{\widehat{ \bf 1}}}
\newcommand{\hg}{{\bf {\widehat{g}}}}
\newcommand{\g}{{\bf g}}
\newcommand{\w}{{\bf w}}
\newcommand{\q}{{\bf q}}
\newcommand{\x}{{\bf x}}
\newcommand{\y}{{\bf y}}

\newcommand{\z}{{\bf z}}
\newcommand{\uu}{{\bf u}}
\newcommand{\vv}{{\bf v}}

\newcommand{\p}{{\bf p}}

\usepackage{tikz,fullpage}
\usepackage{tkz-berge}
\usepackage{verbatim}
\usetikzlibrary{arrows,shapes}

\def\0{{\bf 0}}
\def\1{{\bf 1}}

\def\R{\mathbb{R}}

\title{Linear Time Average Consensus on Fixed Graphs and Implications for Decentralized Optimization and Multi-Agent Control}

\author{Alex Olshevsky\thanks{Coordinated Science Laboratory, University of Illinois at Urbana-Champaign, {\tt aolshev2@illinois.edu}. This research was supported by NSF under award CMMI-1463262 and AFOSR under award FA-9550­15­10394. A preliminary version of this paper was presented at the 5th IFAC Workshop on Distributed Estimation and Control in Networked Systems in 2015.}  }

\begin{document}

\maketitle

\begin{abstract} We describe a protocol for the average consensus problem on any fixed undirected graph whose convergence time scales linearly in the total number nodes $n$.  The protocol is completely distributed, with the exception of requiring all nodes to know the same upper bound $U$ on the total number of nodes which is correct within a constant multiplicative factor. 

We next discuss applications of this protocol to problems in multi-agent control connected to the consensus problem. In particular, we describe protocols for formation maintenance and leader-following with convergence times which also scale linearly with the number of nodes.

Finally, we develop a distributed protocol for minimizing an average of (possibly nondifferentiable) convex functions $ (1/n) \sum_{i=1}^n f_i(\theta)$, in the setting where only node $i$ in an undirected, connected graph knows the function $f_i(\theta)$. Under the same assumption about all nodes knowing $U$, and additionally assuming that the subgradients of each $f_i(\theta)$ have absolute  values upper bounded by some constant $L$ known to the nodes, we show that after $T$ iterations our protocol has error which is $O(L \sqrt{n/T})$. 

\end{abstract}

%\begin{keywords} consensus protocols, multi-agent systems, distributed control, opinion dynamics
%\end{keywords}

%\begin{AMS} 93A14, 93C55, 68Q85
%\end{AMS}

\pagestyle{myheadings}
\thispagestyle{plain}
%\markboth{TEX PRODUCTION AND V. A. U. THORS}{SIAM MACRO EXAMPLES}

\section{Introduction}  The main subject of this paper is the average consensus problem, a canonical problem in multi-agent control: there are $n$ agents, with each agent $i=1, \ldots, n$ maintaining a value $x_i(t)$, updated as a result of interactions with neighbors in some graph $G$, and the agents would like all $x_i(t)$ to approach the initial average $\overline{x} = (1/n) \sum_{j=1}^n x_j(1)$. There is much work on the consensus problem, both classical and recent \cite{degroot, tsi, cybenko, morse}, driven by a variety of applications in distributed computing and multi-agent control. 

Research on consensus protocols is motivated by emerging applications of autonomous vehicles, UAVs, sensors platforms, and multi-agent swarms in general. Protocols designed for such systems should be distributed, relying only on interactions among neighbors, and resilient to failures of links. 
The design of  protocols with these properties often relies on protocols for the consensus problem.

Indeed, as examples we mention recent work on coverage control \cite{gcb08}, formation control \cite{ofm07, othesis},
distributed estimation \cite{XBL05, XBL06, fagnani1, fagnani2, ribeiro2}, cooperative learning \cite{jad-learning}, and Kalman filtering \cite{kalman1, kalman2, kalman3}, distributed task assignment \cite{CBH09},  and distributed optimization \cite{tsi, four};
these papers and others design distributed protocols either by a direct reduction to an appropriately defined consensus problem or by using consensus 
protocols as a subroutine. 

In this paper, our focus is on designing consensus protocols which converge to the initial average as fast as possible. Specifically, we will be concerned with the number of updates until every $x_i(t)$ is within $\epsilon$ of the initial average $\overline{x}$. We refer to this quantity as  the convergence time of the consensus protocol. Note that convergence time is a function of $\epsilon$ and can also depend on other problem parameters as well, for example the initial values $x_i(1), i = 1, \ldots, n$, the communication graph $G$, and the total number of nodes $n$. {\em Our goal is to design protocols with fast convergence times and to exploit them in some of the application areas which rely on consensus, namely in distributed optimization, formation control, and leader-following. }

\subsection{Previous work on convergence time of average consensus\label{sec:constheorem}} 

This question of understanding the convergence time of consensus schemes has received considerable attention. A number of papers studied the convergence speed of discrete-time (not necessarily average) consensus for a wide class of updates over {\em time-varying directed graphs}; see \cite{tsi, morse, cdc05, baa, angeli, csm, cma, jb}. The worst-case bounds on convergence times derived in these papers tended to scale exponentially in $n$; that is, the time until every node was within some small $\epsilon$ of $\overline{x}$ could grow exponentially in the total number of nodes. 

The first convergence time which was always polynomial was obtained in \cite{sicon}, where a protocol was proposed with time  $O \left( n^3 \log (n ||x(1) - \overline{x}||_2/\epsilon) \right)$ until every node was within $\epsilon$ of the average $\overline{x}$ on any connected undirected graph $G$. In fact, this cubic convergence time was shown for time-varying graph sequences which satisfied a long-term connectivity condition. A better analysis of the same protocol in \cite{four}
obtained an improved quadratic convergence time  of $O \left( n^2 \log (n ||x(1) - \overline{x}||_2/\epsilon) \right)$ in the same setting.

The above convergence times were obtained via consensus protocols where each node updated by setting $x_i(t+1)$ to some 
convex combination of its own value $x_i(t)$ and the values $x_j(t)$ of its neighbors. A parallel line of research considered using 
memory to speed up the convergence time of consensus schemes, specifically by setting $x_i(t+1)$ to be a linear combination of $x_i(t), x_i(t-1)$ and the 
values of the neighbors $x_j(t), x_j(t-1)$. Within the context of distributed computation, the first paper, to our knowledge, to make use of this insight was \cite{muthu} which applied it to design load-balancing protocols by a communicating network of processors in the mid-90s. Within the context of consensus, the first paper to make use of this idea was \cite{cao} with later literature in \cite{acce1, acce2, acce3, acce4, acce5, acce6, acce7, acce8, acce9, sarlette} exploring methods of this type. 
%We mention in particular \cite{acce8} which considered speeding up pairwise updates and \cite{acce7} which discussed optimal parameter selection.  

Unfortunately, much of the literature on the subject has update rules depending on the eigenvalues of the graph and thus effectively requiring all the nodes to know the graph $G$. 
Furthermore, none of the consensus protocols which used memory were able to rigorously improve on the worst-case quadratic convergence time for consensus from \cite{four} in the worst case.

On the other hand, moving away from simple schemes where nodes take linear combinations at each step does allow one to speed up convergence
time. For example, \cite{goncalves} demonstrated that if nodes store the values they have received in the past and repeatedly compute
ranks and kernels of matrices which could be $\Omega(n) \times \Omega(n)$, then each node can compute the consensus value in linear time.  Of course, a downside of such schemes is that they are more computationally intensive than the earlier protocols for average consensus which updated at each step by taking convex combinations. 
A related observation is that additional knowledge about the graph on the part of many nodes can lead to fast convergence times in many cases, see e.g., \cite{ang-liu, guillaume, delv}.

On the other hand, linear and finite convergence times can be achieved in any undirected graph by designating one node as the leader and sending all the values to the leader along a spanning tree; the leader can then compute the answer and forward the result back along the same spanning tree. Note that electing a leader and building a spanning tree with the leader as the root can be done in a randomized distributed way in essentially as many iterations as the graph diameter with high probability; see \cite{afek} for precise technical statements. Moreover, such schemes could be made quite light-weight if each node forwards up only an average of the values in its subtree as well as the number of nodes on which this average is based. We refer the reader to \cite{pavone1, pavone2} for schemes along these lines which additionally optimize the number of messages exchanged in the network.  However,  schemes based on these ideas stop working altogether if the communication graphs are arbitrarily time-varying. By contrast, consensus schemes relying on distributed interactions with nearest neighbors, such as the ones we study in this paper, are more robust to unpredictable graph changes. 

%Although the theoretical results of this paper are only for fixed graphs, simulations appear to show linear scaling for variations of Eq. (\ref{alm}) on time-%varying graphs as well (see Section \ref{sec:simul} for details). 

\subsection{Distributed optimization of an average  of convex functions\label{sec:optintro}} 

We will devote a large fraction of this paper to studying a generalization of the average consensus problem which is central within the field of decentralized optimization: there are $n$ nodes in an undirected, connected graph, with node $i$ being the only node which knows the convex function $f_i: \R \rightarrow \R$, and the nodes desire to collectively agree on a minimizer of the average function $f(\theta) = (1/n) \sum_{i=1}^n f_i(\theta)$. 

That is, we would like each node $i$ to maintain a state variable $\theta_i(t)$, updated through interactions with neighbors in some graph $G$ and evaluations of gradients of its convex function $f_i(\theta)$, such that all the $\theta_i(t)$ asymptotically aproach the same minimizer of $f(\theta)$. Note that if we set $f_i(\theta) = (\theta - \theta_i)^2$ for some collection of numbers $\theta_1, \ldots, \theta_n$, the optimal solution of the decentralized optimization problem is simply the average of the numbers $\theta_1, \ldots, \theta_n$. Thus this problem contains average consensus
as a special case. 

A number of problems can be cast within the framework of distributed optimization. For example, if $f_i(\theta) = |\theta - \theta_i|$, then the optimal solution is 
the median of the numbers $\theta_1, \ldots, \theta_n$.  For specific application areas where this problem appears, we mention  robust statistical inference \cite{rabbat}, non-autonomous power control \cite{ram_info}, network resource allocation \cite{beck-angelia},
distributed message
routing \cite{neglia}, and spectrum access coordination \cite{li-han}.

The first rigorous analysis of this  problem was given in \cite{NO} and the problem has received considerable attention in recent years 
due its repeated appearance throughout distributed computing and multi-agent control; we mention \cite{NO, NOP, LoO, rabbat, ram_info, jo1, jo2, tsia1, tsia2, tsia3, ermin, bianchi1, bianchi2, bianchi3, ribeiro1,  joh1, yin, chen}, among many others. 

Indeed, the amount of work on this problem is so vast that we do not attempt a general overview, and only discuss previous work on convergence times which is our main focus.  For this problem, convergence time is often defined to be the number of updates until the objective is within $\epsilon$ of its value at the minimizer(s). Focusing on algorithms which assume that, at every step, every node $i$ can query the subgradient of the function $f_i(\cdot)$, and assuming only that the functions are convex, the protocol with the best known convergence time was proposed in \cite{RNV} (building on the previous paper \cite{NO}). Under the assumptions that (i) every node begins with an initial condition within a ball of fixed size of some optimizer $w^*$ (ii) the subgradients of all $f_i(\theta)$ are at most $L$ in absolute value, we have that
optimizing the step-size bounds in \cite{RNV} yields an $O(L^2 n^4/\epsilon^2)$ time until the objective is within $\epsilon$ of its optimal value.  We also mention the follow-up paper \cite{DAW} which showed some improvements on the convergence rates of \cite{RNV} for many specific graphs under the stronger assumption that each node knows the spectral gap of the graph. 

We note that better convergence rates are available under stronger assumptions. For example, under the assumption that the functions $f_i(\cdot)$ are strongly convex, a geometrically convergent algorithm was given in \cite{yin}. Moreover, under the assumption that, at each step, 
every node $i$ can compute a minimizer of  $f_i(\theta)$ plus a quadratic function, further improved rates have been derived. For example, a distributed version of the Alternating Direction Method of Multipliers (ADMM) relying on such updates was analyzed in \cite{ermin} and \cite{bianchi1, ribeiro1}. The former paper showed $O(1/\epsilon)$ time until the objective is within an $\epsilon$ of its optimal value under only the assumption of convexity on the functions $f_i(\cdot)$; the latter two papers showed geometric $O(\log 1/\epsilon)$ convergence under strong convexity. Finally, the recent preprint \cite{asu-recent} studies how geometric convergence time of the ADMM for strongly convex function depends on the network parameters and the condition numbers of the underlying functions.

\subsection{Our results and the structure of this paper}  Under the assumption that the graph $G$ is undirected and connected, and under the further assumption that there is an upper bound $U$ on the number of nodes $n$ which is within a constant multiplicative factor of $n$ and known to all the nodes\footnote{Specifically, to obtain linear convergence time we will need $cn \geq U \geq n$ to hold, where $c$ is some constant independent of $n$.}, {we provide  a protocol whose convergence time scales {linearly} with  the number of nodes $n$.}

This is an improvement over the previously best-known scaling, which was quadratic in $n$ \cite{four}; however, note that our result here  is established under a stronger assumption, namely the nodes knowing the upper bound $U$. 

This additional assumption is satisfied if, for example, the total number of agents is known.  More broadly, the availability of such a $U$ only requires the nodes to know roughly  the total number of nodes in the system, and will be satisfied whenever reasonable bounds on network size are available. For example, consider the scenario when the system starts with a fixed and known number of nodes $n_0$ and where nodes may drop out of the network due to faults. In such a setting, it would take over $80\%$ of the nodes to fail before $n_0$ would cease to be a multiplicative-factor-of-five approximation to the total number of nodes.  Similar statements can be made in more complex scenarios, for example when nodes are allowed to both join and leave the system.

Besides the result on average consensus, this paper has three other contributions. Most prominently, we provide a new protocol for the problem of distributed optimization of  an average  of (possibly nondifferentiable) scalar convex functions, and show that, under appropriate technical assumptions, its convergence time is $O(n/\epsilon^2)$. This convergence time follows as a consequence of our result on linear time average consensus. Recalling our discussion in Section \ref{sec:optintro}, the best previous comparable convergence time was $O(n^4/\epsilon^2)$. However, once again our improved convergence time is established under a stronger assumption, namely the nodes knowing the bound $U$. 

Finally, we will aso discuss the problems of formation control from offset measurements and leader-following, to be formally defined later within the body of the paper. For both of these problems, we will show that protocols with linear convergence times can be given as a consequence of our result on linear time consensus.

We now outline the remainder of this paper. Section \ref{sec:consproof} is dedicated to stating and proving our result on linear time consensus as well as to a discussion of some of its variations.  Section \ref{sec:opt} is dedicated to the problem of decentralized optimization of an average of convex functions.  Section \ref{sec:multi} formally describes the problems of formation control from offset measurements and leader-following and proposes protocols for these problems with a linear convergence time.  Section \ref{sec:simul} contains some simulations and the conclusion of the paper may be found in Section \ref{sec:concl}.

\subsection{Notation} 

We will use the shorthand $[n]$ to denote the set $\{1, \ldots, n\}$. 
We  use the standard notation $N(i)$ to denote the set of neighbors of node $i$ in
$G$; note that in contrast to some of the papers on the subject, we will assume $i \notin N(i)$ for all $i$. The degree of node $i$ is denoted by $d(i)$.  We follow the convention of bolding vectors while scalars remain unbolded. Thus the $i$'th coordinate of a vector $\x \in \R^n$ is denoted by $x_i$. We will also sometimes use $[\x]_i$ to denote the same $i$'th coordinate, and similarly $[A]_{i,j}$ will be used to denote the $i,j$'th entry of the matrix $A$. For a vector $\x$, the notation $[\x]_{1:k}$ denotes the vector in $\R^k$ obtained by taking the first $k$ coordinates of $\x$. The symbol $\1$ stands for the all-ones vector.

\section{Linear time consensus\label{sec:consproof}} The purpose of this section is to propose a new protocol for the average consensus problem and to prove this protocol has convergence 
which is linear in the number of nodes in the network. We start with a description of the protocol itself. 

\subsection{The protocol}

Each node $i$ in a fixed undirected graph $G$ maintains the variables $x_i(t), y_i(t)$ initialized as $y_i(1)=x_i(1)$ and updated as: \begin{align} y_i(t+1)  & =  x_i(t) + \frac{1}{2} \sum_{j \in N(i)} \frac{x_j(t) - x_i(t)}{\max(d(i), d(j))} \nonumber \\
x_i(t+1) & = y_i(t+1) + \left(  1 - \frac{2}{9U+1} \right) \left( y_i(t+1) - y_i(t) \right) \label{alm}
\end{align}  where $N(i)$ is the set of neighbors of node $i$ in $G$ and $d(i)$ is the degree of node $i$. As previously discussed, $U$ can be any number that satisfies $U \geq n$, where $n$ is the number of nodes. 

\smallskip

Before discussing our results, let us briefly describe some intuition for this protocol. Observe that the first line above is a form of the usual consensus update where $y_i(t+1)$ is set to a convex combination of $x_i(t)$ and $x_j(t)$ for neighbors $j$ of $i$. Note, however, that nodes $i$ and $j$ place a weight on each other's values inversely proportional to the larger of their degrees. This is a variation on so-called Metropolis weights, first introduced in  \cite{XBL06}; informally speaking,  placing smaller weights on neighbors with high degree tends to make highly connected cliques of nodes less resistant to changing their values and  speeds up convergence time. 

The second line of the protocol performs an extrapolation step by setting $x_i(t+1)$ to be a linear combination of $y_i(t+1)$ and the previous value $y_i(t)$; speaking informally once again, this linear combination has the effect of adding ``momentum'' by ``pushing'' $x_i(t+1)$ in the direction of the difference of iterates $y_i(t+1)-y_i(t)$. There is no intuitive explanation (as far as we know) for why the addition of such a momentum term speeds up linear iterations, although non-intuitive proofs of speed-up due to the addition of momentum terms are available in many contexts -- see, for example, the references discussed in  Section \ref{sec:constheorem}. 

\subsection{Convergence time}    We now state our convergence time result.

\medskip

\begin{theorem} Suppose each node in an undirected connected graph $G$ implements the update of Eq. (\ref{alm}). 
If $U \geq n$ then we have $$||\y(t) - \overline{x} \1||_2^2 \leq 2 \left( 1 - \frac{1}{9U} \right)^{t-1} ||\y(1) - \overline{x} \1||_2^2.$$  \label{constheorem}
\end{theorem} 

\medskip

Note that under the assumption that $U$ is in fact within a constant multiplicative factor of $n$, this implies a $O \left(n \ln \left( {||\y(1) - \overline{x} \1||_2}/{\epsilon} \right) \right)$ convergence time until $||\y(t) - \overline{x} \1||_2$ is below $\epsilon$, which is the linear scaling of the title of this paper. Note that since the infinity-norm of a vector is upper bounded by the  two-norm, we have that $||\y(t) - \overline{x} \1||_{\infty}$ is below $\epsilon$ after this many iterations as well. In a sense, this convergence time is close to optimal, since a graph of $n$ nodes may have diameter $n-1$, implying that a consensus protocol which brings all nodes close to the average in fewer than $n-1$ steps is not possible.

%As we have already discussed, Theorem \ref{constheorem} may be viewed as an improvement of the literature described previously, where the best convergence time from \cite{four} was quadratic in %the number of nodes. Note, however, that Theorem \ref{constheorem} requires a fixed interconnection graph $G$, whereas the quadratic convergence time of  of \cite{four} can be achieved on %time-%varying sequences of undirected graphs satisfying a long-term connectivity 
%condition. Furthermore, our protocol requires all nodes to know the same upper bound $U$ on the total number of nodes, and this bound must be correct within a constant factor; by contrast, no global %knowledge is required in the protocol of \cite{four}.

%Compared to the results of \cite{goncalves}, Theorem \ref{constheorem} is weaker since consensus is achieved only in a limiting sense (by contrast \cite{goncalves} achieves exact %consensus in linear and finite time). However, the advantage of our result is that the update of Eq. (\ref{alm}) requires very little computational power to execute on the part of the nodes %and consequently is nicely scalable in large networks.

%In this section, we assume the vectors $\x(t), \y(t)$ are generated according to the accelerated Metropolis method. 

\subsection{Proof of linear convergence time} 
The remainder of this section is dedicated to proving Theorem \ref{constheorem}. We begin  by introducing some additional notation. 

Recall that there is a fixed, undirected connected graph $G = ([n], E)$. We define the Metropolis matrix $M$ to be  the unique stochastic matrix with off-diagonal entries
\[ [M]_{ij} = \begin{cases} \frac{1}{ \max( d(i), d(j))} & \mbox{ if } (i,j) \in E. \\ 0 & \mbox{ if } (i,j) \notin E. \end{cases} \] Note that the above equation specifies all the off-diagonal entries of $M$, which uniquely define the diagonal entries due to the requirement that $M$ be stochastic. Furthermore, we define the {\em lazy Metropolis matrix} $M'$ as \[ M' = \frac{1}{2} I + \frac{1}{2} M, \] where $I$ is the $n \times n$ identity matrix. Note that $M'$ is stochastic, symmetric, and diagonally dominant. By stochasticity, its largest eigenvalue is $1$, and diagonal dominance implies that all of its eigenvalues are nonnegative. We will use $\lambda_2(M')$ to denote the second-largest eigenvalue of $M'$.

Finally, given a Markov chain with probability transition matrix $P$, the hitting time $\H_P(i \rightarrow j)$ is defined to be the expected number of steps until the chain reaches $j$ starting from $i$.  Naturally, $\H_P(i \rightarrow i)=0$ for all $i \in [n]$.  A vector is called stochastic if its entries are nonnegative and add up to $1$. The total variation distance between two stochastic vectors $p,q$ of the same size is defined to be $||p-q||_{\rm TV} = (1/2) ||p-q||_1$.  Given a Markov chain $P$ with stationary distribution $\pi$, the mixing time $t_{\rm mix}(P, \epsilon)$ is defined to be the smallest $t$ such that for all stochastic vectors $p_0$, \[ ||p_0^T P^t - \pi||_{\rm TV} \leq \epsilon \] 
%Moreover, following \cite{peres}, we will find it useful to deal with the quantity
%$\overline{d}(t)$ defined as \[ \overline{d}(t) = \max_{i,j \in [n]} ||e_i^T P^t - e_j^T P^t||_{\rm TV}\] 

With these definitions in place we now proceed to our first lemma, which states that the lazy Metropolis chain has an eigenvalue gap which is at least quadratic in  $1/n$. 

\medskip

\begin{lemma} \[ \lambda_2(M') < 1 - \frac{1}{71 n^2} \] \label{eigbound}
\end{lemma} \begin{proof} Our starting point is the following result, proved in \cite{nonaka}: \[ \max_{i,j \in [n]} \H_M(i \rightarrow j) \leq 6 n^2 \] Each step in the lazy Metropolis chain $M'$ may be viewed as a step in $M$ after a waiting time which is geometrically distributed with parameter $1/2$. It follows that all hitting times in $M'$ are at most double the hitting times in $M$ and therefore \[ \max_{i,j \in [n]} \H_{M'}(i \rightarrow j) \leq 12 n^2 \] We next make use of the fact that for diagonally dominant reversible Markov chains the mixing time can be bounded in terms of hitting time; more precisely,  Eq. (10.23) in \cite{peres} says that for any reversible Markov chain $P$ with stationary distribution $\pi$ such that $[P]_{ii} \geq 1/2$, \[ t_{\rm mix} \left(P, \frac{1}{4} \right) \leq 2 \max_{j \in [n]} \sum_{i \in [n]} \pi_i H(i \rightarrow j) + 1 \] For reasons of convenience, we wish to instead upper bound $t_{\rm mix}(P, 1/8)$. Following the proof in \cite{peres} yields the inequality 
 \[ t_{\rm mix} \left(P, \frac{1}{8} \right) \leq 8 \max_{j \in [n]} \sum_{i \in [n]} \pi_i H(i \rightarrow j) + 1 \]  
 Since $\sum_{i=1}^n \pi_i = 1$, this implies 
\[ t_{\rm mix} \left(P, \frac{1}{8} \right) \leq 8 \max_{i,j \in [n]} \H(i \rightarrow j) + 1 \] It therefore follows that the lazy Metropolis walk (which is clearly reversible, since $M'$ is a symmetric matrix) satisfies
\[ t_{\rm mix} \left(M', \frac{1}{8} \right) \leq 96 n^2 + 1 < 97n^2\] The strict inequality above follows immediately when $n>1$ and holds automatically when $n=1$. Finally, the distance between $\lambda_2(M')$ and $1$ can be upper bounded by the mixing time; indeed, Eq. (12.12) in \cite{peres} implies 
\[ t_{\rm mix}\left(M', \frac{1}{8} \right) \geq \left( \frac{1}{1 - \lambda_2(M')} - 1 \right) \ln 4 \] which in turn implies
\[ \frac{1}{1 - \lambda_2(M')} < 71 n^2  \] which implies the statement of the current lemma. 
\end{proof} 

\medskip
We now shift gears and consider the family of matrices \[ B(\lambda) =  \left[ \begin{array}{cc}  \alpha \lambda  & -(\alpha-1) \lambda \\  1 & 0 \end{array} \right] \] where $\alpha = 2 - 2/(9U+1)$. Note that since $U \geq n \geq 1$ we have that $\alpha \in (1,2)$. Although this appears unmotivated at first glance, the following lemma explains why these matrices are of interest to us. 

\medskip

\begin{lemma} Let $1=\lambda_1, \ldots, \lambda_n$ be the eigenvalues of $M'$ and let $Q$ be an orthogonal matrix whose $i$'th column is an eigenvector corresponding to $\lambda_i$. Let us adopt the convention that $\y(0)=\y(1)$ and consider the change of coordinates $\z(t) = Q^T \y(t)$. We have that 
 \[  \left[ \begin{array}{c} z_i(t+1) \\ z_i(t) \end{array} \right] = B(\lambda_i)   \left[ \begin{array}{c} z_i(t) \\ z_i(t-1) \end{array} \right]  \] \label{coordchange}
\end{lemma} 

\begin{proof} Indeed, 
\begin{eqnarray*} \y(t+1) =  M' \x(t) \
 =  M' \left( \y(t) + \left(1 - \frac{2}{9U+1} \right) (\y(t) - \y(t-1) \right) 
 =  \alpha M' \y(t) - (\alpha - 1) M' \y(t-1) 
\end{eqnarray*} 
so that 
\[ \left[ \begin{array}{c} \y(t+1) \\ \y(t) \end{array} \right] = \left[ \begin{array}{cc} \alpha M' & -(\alpha - 1) M' \\ I & 0 \end{array} \right] \left[ \begin{array}{c} \y(t) \\ \y(t-1) \end{array} \right] \]  It follows that 
\begin{align*} \left[ \begin{array}{c} \z(t+1) \\ \z(t) \end{array} \right] &  =  \left[ \begin{array}{cc}  Q^T & 0 \\  0 & Q^T \end{array} \right]  \left[ \begin{array}{cc} \alpha M' & -(\alpha - 1) M' \\ I & 0 \end{array} \right] \left[ \begin{array}{cc}  Q & 0 \\  0 & Q \end{array} \right] \left[ \begin{array}{c} \z(t) \\ \z(t-1) \end{array} \right]  
\end{align*} 
Letting $\Lambda = {\rm diag}(\lambda_1, \ldots, \lambda_n)$, we may write this as
\[  \left[ \begin{array}{c} \z(t+1) \\ \z(t) \end{array} \right] = \left[ \begin{array}{cc}  \alpha \Lambda & -(\alpha - 1) \Lambda \\  I & 0 \end{array} \right]   \left[ \begin{array}{c} \z(t) \\ \z(t-1) \end{array} \right]  \] This is equivalent to the statement of this lemma.  
\end{proof}

\medskip

In other words, this lemma shows that our update iteration of Eq. (\ref{alm})  can be viewed in terms of  multiplication by $B(\lambda_i)$ once we switch coordinates 
from $\y(t)$ to $\z(t)$. The next lemma relates the distance to the desired final limit $\overline{x} \1$ in the new coordinates. 

\medskip

\begin{lemma} \[ ||\y(t) - \overline{x} \1||_2^2 = \sum_{i =2}^n z_i^2(t) \] \label{zavg}
\end{lemma} 

\begin{proof} Indeed, since $U$ is orthogonal, \[ ||\y(t) - \overline{x} \1||_2^2 = ||U^T \y(t) - U^T \overline{x} \1 ||_2 = ||\z(t) - \sqrt{n} \overline{x} {\bf e_1} ||_2^2 \]  where the last equality follows from the fact that the first row of $U^T$ is $(1/\sqrt{n}) \1$ and all other rows of $U^T$ are orthogonal to $\1$. However, the first element of $\z(t)$ is  $(1/\sqrt{n}) \1^T \y(t)  = \sqrt{n} \overline{x}$ which immediately implies the lemma.  
\end{proof}  

\medskip

The next lemma is the key point of the analysis: it shows that multiplication by $B(\lambda_i)$ converges to zero at a rate which is geometric with constant $1-1/(9U)$ for any eigenvalue $\lambda_i$ of $M'$. Our main result, Theorem \ref{constheorem}, will follow straightforwardly from this lemma. 

\medskip

\begin{lemma} Suppose $U \geq n$.  If $\lambda_i \neq 1$ is an eigenvalue of $M'$, $r$ is some real number, and $\q(t) = (B(\lambda_i))^{t-1} \left[ \begin{array}{c} r \\ r \end{array} \right]$. Then 
$$q_2^2(t) \leq 2r^2\left( 1 - \frac{1}{9U} \right)^{t-1}.$$\label{nonprincipaldecay}
\end{lemma}

\medskip

%As we have already mentioned, once this lemma has been established, the proof of Theorem \ref{constheorem} will be immediate by putting together the results we have already %obtained: we will simply change variables to $\z(t)$ and use Lemma \ref{nonprincipaldecay} to bound the rate at which $||\z(t)||_2$ converges to zero; by Lemmas \ref{coordchange} %and \ref{zavg}, this will correspond to the rate at which $\y(t)$ converges the optimal point $\overline{x} \1$. 

As we have already mentioned, once this lemma has been established the proof of Theorem \ref{constheorem} will be immediate by putting together the results we have already obtained. Indeed, observe that Lemma \ref{coordchange} showed that after changing coordinates to $\z(t)$, our algorithm corresponds to multiplying $z_i(t)$ by $B(\lambda_i)$; and Lemma \ref{zavg} showed that distance between $\y(t)$ and the final limit $\overline{x} \1$ is exactly $\sum_{i \geq 2} z_i^2(t)$. It remains to bound the rate at which the process of repeated multiplication by $B(\lambda_i)$ converges to zero and that is precisely what Lemma \ref{nonprincipaldecay} does.

It is possible to prove Lemma \ref{nonprincipaldecay} by explicitly diagonalizing $B(\lambda)$ as a function of $\lambda$, but this turns out to be somewhat cumbersome. A shortcut is to rely on a classic result of Nesterov on accelerated gradient descent, stated next. 

\medskip

\begin{theorem}[Nesterov, \cite{nesbook}] \label{accelthm} Suppose $f: \R^n \rightarrow \R$ is a $\mu$-strongly convex function, meaning that 
that for all $\x, \y \in \R^n$, \begin{equation} \label{sconvex} f(\w) \geq f(\uu) + f'(\uu)^T (\w - \uu) + \frac{\mu}{2} ||\w - \uu||_2^2 \end{equation} Suppose further that $f$ is differentiable everywhere and its gradient has Lipschitz constant $L$, i.e., \[ ||f'(\w) - f'(\uu)||_2 \leq L ||\w-\uu||_2 \] Then $f$ has a unique global minimum; call it $\uu^*$. We then have that the update 
\begin{eqnarray}  \w(t+1) & = & \uu(t) - \frac{1}{L} f'(\uu(t)) \nonumber \\ 
\uu(t+1) & = & \w(t+1) + \left( 1 - \frac{2}{\sqrt{Q}+1} \right) \left( \w(t+1) - \w(t) \right) \label{accelupdate}
\end{eqnarray}
 initialized at $\w(1)=\uu(1)$ with $Q=L/\mu$ has the associated convergence time bound  \[ f(\w(t)) - f(\uu^*) \leq \left( 1 - \frac{1}{\sqrt{Q}} \right)^{t-1} \left( f(\w(1)) + \frac{\mu}{2} ||\w(1) - \uu^*||_2^2 - f(\uu^*) \right)\]
\end{theorem}

The theorem is an amalgamation of several facts from Chapter 2.2 of \cite{nesbook} and is taken from \cite{bubeck}. We can appeal to it to 
give a reasonably quick proof of Lemma \ref{nonprincipaldecay}. 

\medskip

\begin{proof}[Proof of Lemma \ref{nonprincipaldecay}]  We first write the iteration of multiplication by $B(\lambda)$ in terms of accelerated gradient descent. 
 Indeed,  consider the function $g(x) = (1/2)x^2$ and  let $\lambda$ be any nonprincipal eigenvalue of $M'$. Observe that the derivative of
$g(x)$ is $1$-Lipschitz and therefore, using the fact that $\lambda \in [0,1)$, we obtain that the gradient of $g(x)$ is $L$-Lipschitz for $L = \frac{1}{1-\lambda}$. Furthermore, $g(x)$ is $1$-strongly convex, 
and is consequently $\mu$-strongly convex for any  $\mu \in (0,1]$. Let us choose $\mu = {1}/{(81U^2 (1-\lambda))}$  which works since  by Lemma \ref{eigbound}, 
\[ \mu \leq \frac{71}{81} \frac{n^2}{U^2} \leq 1 \] and therefore $g(x)$ is $\mu$-strongly convex. 

For these choices of $L$ and $\mu$, the parameter $Q=L/\mu$ equals
\[ Q = { \frac{1}{(1-\lambda) \mu}} = 81U^2 \] and the accelerated gradient descent iteration on $g(x)$ takes the form \begin{eqnarray*} w(t+1) & = & u(t) - \frac{1}{1/(1-\lambda)} u(t) \\ 
 u(t+1) & = &  \left( 2 - \frac{2}{\sqrt{Q}+1} \right) w(t+1) - \left( 1 - \frac{2}{\sqrt{Q}+1} \right) w(t) 
\end{eqnarray*} This may be rewritten as 
\begin{eqnarray*} w(t+1) & = & \lambda u(t) \\ 
u(t+1) & = & \left( 2 - \frac{2}{9U+1} \right) w(t+1) - \left( 1 - \frac{2}{9U+1} \right) w(t) 
\end{eqnarray*}  Adopting the convention that $w(0) = w(1)$, we can write \begin{small}
\[ w(t+1) = \lambda u(t) = \lambda \left( \left( 2 - \frac{2}{9U+1} \right) w(t) - \left( 1  - \frac{2}{9U+1}  \right)  w(t-1) \right). \] \end{small}
In matrix form, we have
\[ \left[ \begin{array}{c} w(t+1) \\ w(t) \end{array} \right] = B(\lambda) \left[ \begin{array}{c} w(t) \\ w(t-1) \end{array} \right], \]
so that
\[ \left[ \begin{array}{c} w(t) \\ w(t-1) \end{array} \right] = (B(\lambda))^{t-1} \left[ \begin{array}{c} w(1) \\ w(0) \end{array} \right]. \]
Thus by Theorem \ref{accelthm}, we have that 
\[ \frac{1}{2} w^2(t) \leq \left( 1 - \frac{1}{9U} \right)^{t-1} \left( \frac{1}{2} w^2(1) + \frac{\mu}{2} w^2(1) \right)  \]
which implies the statement of the lemma. 
\end{proof}

\medskip

With these reults established, we now turn to proof of Theorem \ref{constheorem}. As we have already mentioned, this just involved putting the previous lemmas together. 

\medskip

\begin{proof}[Proof of Theorem \ref{constheorem}] Indeed, 
\begin{eqnarray*} ||\y(t) - \overline{x} \1||_2^2 & = &  \sum_{i=2}^n z_i^2(t)  \
 \leq  2 \left( 1 - \frac{1}{9U} \right)^{t-1} \sum_{i=2}^n z_i^2(1) 
 =  2 \left( 1 - \frac{1}{9U} \right)^{t-1} ||\y(1) - \overline{x} \1||_2^2 
\end{eqnarray*} where the first line used  Lemma \ref{zavg} and second line used both Lemma \ref{coordchange} and Lemma \ref{nonprincipaldecay}. This proves the theorem. 
\end{proof} 

\subsection{Two-dimensional grids and geometric random graphs\label{grid}} We now make the observation that the convergence rate of Theorem \ref{constheorem} can be considerably improved for two classes of graphs which often appear in practice, namely two-dimensional grids and geometric random graphs. 

A 2D grid is the graph on $n=k^2$ nodes where each node is associated with a coordinate $(i,j)$ where $i$ and $j$ are integers between $1$ and $k$. Two nodes $(i,j)$ and $(i',j')$ and connected by an edge if and only if $|i - i'| + |j-j'| = 1$. 

A geometric random graph $G(n,r)$ is formed by placing $n$ nodes at uniformly random positions in the square $[0,1]^2$ and putting an edge between two nodes if and only if the Euclidean distance between them is at most $r$. It is known that if the connectivity radius $r$ satisfies $r^2 \geq c (\log n)/n$ for some (known) constant $c$,  then the geometric random graph will be connected with high probability\footnote{We follow the usual convention of saying that a statement about graphs with $n$ nodes holds with high probability if the probability of it holding approaches $1$ as $n \rightarrow \infty$.}
 \cite{avin}.

Both grids and geometric random graphs are appealing models for distributed systems. Neither one has any links that are ``long range.'' Whereas an Erdos-Renyi random graph will have much more favorable connectivity properties, it is challenging to use it as a blueprint for, say, a layout of processors, since regardless of where the nodes are placed there will be links between processors which are not geographically close. Grids and geometric random graphs do not have this problem. Moreover, we further note that geometric random graphs are a common model of wireless networks 
(see e.g., \cite{geogwire}).

In this section, we state and sketch the proof of a proposition which shows that a version of Eq. (\ref{alm}) will have convergence time which scales essentially as {\em the square root of the number of nodes} in grids and geometric random graphs. Intuitively, both of these graph classes have diameters which are substantially less than $n$, and consequently an improvement of the worst-case convergence time is possible. 

\medskip

\begin{proposition} \label{fasterprop} Consider the update rule 
 \begin{eqnarray} y_i(t+1)  & = & x_i(t) + \frac{1}{2} \sum_{j \in N(i)} \frac{x_j(t) - x_i(t)}{\max(d(i), d(j))} \nonumber \\
x_i(t+1) & = & y_i(t+1) + \left(  1 - \gamma \right) \left( y_i(t+1) - y_i(t) \right) 
\end{eqnarray} and fix $U \geq n$.  There is a choice of $\gamma = \Theta \left( \frac{1}{\sqrt{U \log U}} \right)$ such that if $G$ is a grid on $n$ nodes, then \begin{equation} \label{fasterconv} ||\y(t) - \overline{x} \1||_2^2 \leq 2 \left( 1 - \gamma \right)^{t-1} ||\y(1) - \overline{x} \1||_2^2 \end{equation} Moreover, if 
$G$ is a geometric random graph $G(n,r)$ with $r^2 \geq \frac{8c \log n}{n}$ where $c>1$, then with high probability we have that Eq. (\ref{fasterconv}) holds for all $t \geq 1$.
\end{proposition}

Note that under the assumption that $U$ is within a constant factor of $n$, the above proposition gives a convergence time of 
$O \left( \sqrt{n \log n} \ln \left( {||\y(1) - \overline{x} \1||_2}/{\epsilon} \right) \right)$ until $||y(t) - \overline{x} \1||_{2}$  is at most 
$\epsilon$. 

\begin{proof}[Proof sketch of Proposition \ref{fasterprop}] The proof parallels the proof of Theorem \ref{constheorem} and we only sketch it here. The key point is that our starting point for the proof of Theorem \ref{constheorem} was the bound ${\cal H}_M(i,j) \leq 6n^2$, proved in \cite{nonaka}; however, in both the grid and the geometric random graph the better bound
 ${\cal H}_M(i,j) \leq O(n \log n)$, is available. In the 2D grid, this can be obtained by putting together the arguments of  Propositions 9.16 and 10.6 of \cite{peres}. For the geometric random graph under the assumed lower bound on $r$, this is true with high probability as a consequence of Theorem 1 of \cite{avin}. 

The rest of the argument simply reprises verbatim the proof of Theorem \ref{constheorem}.   Indeed, the argument of Lemma \ref{eigbound} repeated verbatim now leads to the estimate $\lambda_2(M') \leq 1 - 1/\Omega( n \log n)$. We then proceed as before until reaching Lemma \ref{nonprincipaldecay}, when instead of taking $\mu = 1/((1-\lambda) 81U^2)$ we take $\mu = 1/((1-\lambda) \Theta(U \log U))$. This leads to a value of $Q = L/\mu = \Theta( U \log U )$ and consequently $\sqrt{Q} = \Theta \left( \sqrt{U \log U} \right)$. Application of Nesterov's Theorem \ref{accelthm} 
then proves the current proposition. 
\end{proof}

\section{Linear Time Decentralized Optimization\label{sec:opt}}

In this section, we describe how we can use the the convergence time result of Theorem \ref{constheorem} to design a linear time protocol for the decentralized optimization problem 
\[ \min_{\theta \in \R}  ~ \frac{1}{n} \sum_{i=1}^n f_i(\theta), \] in the setting when only node $i$ in an undirected connected graph $G$ knows the convex function $f_i(\cdot)$. The goal is for every node to 
maintain a variable, updated using a distributed algorithm, and these variables should all converge to the same minimizer of the above optimization problem. 

This problem goes by the name of distributed optimization or decentralized optimization, following \cite{NO} which provided the first rigorous analysis of it. More sophisticated versions have been studied, for example by introducing constraints at each node or couplings between nodes in the objective, but here we will restrict our analysis to the version above. 

\subsection{The protocol} Each node $i$ in the fixed, undirected graph $G$ starts with initial state $x_i(1)$, and besides maintaining the variable $x_i(t)$, it maintains two extra variables $z_i(t)$ and $y_i(t)$. Both $z_i(1)$ and $y_i(1)$ are initialized to $x_i(1)$, and the variables are updated as 

\begin{eqnarray} y_i(t+1) & = & x_i(t) +  \frac{1}{2} \sum_{j \in N(i)} \frac{x_j(t) - x_i(t)}{\max(d(i), d(j))} - \beta g_i(t) \nonumber \\ 
z_i(t+1) & = & y_i(t) - \beta g_i(t)  \label{optaccel} \\ 
x_i(t+1) & = & y_i(t+1) + \left(  1 - \frac{2}{9U+1} \right) \left( y_i(t+1) - z_i(t+1) \right), \nonumber
\end{eqnarray}  where $g_i(t)$ is the subgradient of $f_i(\theta)$ at $\theta = y_i(t)$ and $\beta$ is a step-size to be chosen later.  

\medskip

We briefly pause to describe the motivation behind this scheme. This method builds on the distributed subgradient method 
\begin{equation} \label{distsubg} x_i(t+1) = \sum_{j \in N(i)} a_{ij}(t) x_j(t) - \alpha g_i(t), \end{equation} where $[a_{ij}(t)]$ is usually taken to be a doubly stochastic matrix and $g_i(t)$  a subgradient of the 
function $f_i(\theta)$ at the point $\theta = x_i(t)$. The above method was first rigorously analyzed in \cite{NO}. Intuitively, every 
node takes a subgradient of its own function and these subgradients are then ``mixed together'' using the standard consensus scheme of multiplication by the
doubly stochastic matrices $[a_{ij}(t)]$. 

The method of Eq. (\ref{optaccel}) proceeds via the same intuition, only replacing the usual consensus scheme with the linear-time consensus scheme from Section \ref{sec:consproof}. Since the linear time consensus scheme keeps track of two variables, which must be initalized to equal each other, each new subgradient has to enter the update Eq. (\ref{optaccel}) twice, once in the update for $y_i(t+1)$ and once in the update for $z_i(t+1)$. As a result, the final scheme has to maintain three variables, rather than the single variable of Eq. (\ref{distsubg}).

\subsection{Convergence time} 

We begin by introducing some notation before stating our result on the performance of this scheme. We will use ${\cal W^*}$ to denote the set of global minima of $f(\theta)$. We will use the standard hat-notation for a running average, e.g., 
 $\widehat{y}_i(t) = (1/t) \sum_{k=1}^t y_i(k)$. To measure the convergence speed of our protocol, we introduce two measures of performance. One is the dispersion of a set of points, intuitively measuring how far from each other the points are, \[ {\rm Disp}(\theta_1, \ldots, \theta_n) = \frac{1}{n} \sum_{i=1}^n \left| \theta _i - {\rm median}(\theta_1, \ldots, \theta_n) \right|. \] The other is the ``error'' corresponding to the function $f(\theta) = (1/n) \sum_{i=1}^n f_i(\theta)$, namely
\[ {\rm Err}(\theta_1, \ldots, \theta_n) = \left( \frac{1}{n} \sum_{i=1}^n f_i (\theta_i) \right) - f(w^*), \] where $w^*$ is any point in the optimal set ${\cal W}^*$. Note that to define ${\rm Err}(\theta_1, \ldots, \theta_n)$ we need to assume that the set of global minima ${\cal W}^*$ is nonempty. For the remainder of this section, we will assume in this section that $\x(t), \y(t), \z(t)$ are generated according to Eq. (\ref{optaccel}). 

With these definitions in place, we have the following theorem on the convergence time of the method of Eq. (\ref{optaccel}).

\medskip

\begin{theorem} Suppose that $U \geq n$ and $U=\Theta(n)$, that $w^*$ is a point in ${\cal W}^*$, and that the absolute value of all the subgradients of all $f_i(\theta)$ is bounded by some constant $L$. If every node in an undirected connected graph $G$ implements the update of Eq. (\ref{optaccel}) with the step-size $\beta = \frac{1}{L \sqrt{UT}}$ we then have  \begin{eqnarray}
{\rm Disp}(\widehat{y}_1(T), \ldots, \widehat{y}_n(T)) & = & O \left(  \sqrt{\frac{n}{T}} \left(  \frac{||\y(1) - \overline{x} \1||_2}{\sqrt{T}} + 1 \right) \right)  \label{meddecay} \\
 {\rm Err}(\widehat{y}_1(T), \ldots, \widehat{y}_n(T)) & = & O \left( L \sqrt{\frac{n}{T}}  \left(  \frac{||\y(1) - \overline{x} \1||_2}{\sqrt{T}} + 1 + (\overline{x} - w^*)^2 \right) \right)  \label{funcdecay}
\end{eqnarray}  \label{opthm}
\end{theorem}

\medskip

Note that the time it takes for the bounds of Theorem \ref{opthm} to shrink below $\epsilon$ scales as  $O(n/\epsilon^2)$ with $n$ and $\epsilon$, which is an improvement of the previously best $O(n^4/\epsilon^2)$ scaling from \cite{RNV}. However, as we remarked in the introduction, Theorem \ref{opthm} is for a fixed undirected interconnection graph, whereas the previous analyses of the distributed subgradient method as in \cite{NO, RNV} worked for time-varying undirected graphs satisfying a long-term connectivity condition. Furthermore, Theorem \ref{opthm} requires all nodes knowing the same upper bound $U$ which is within a constant factor of the total number of nodes in the system, which was not required in \cite{NO, RNV}.

\subsection{Proof of linear convergence time} The remainder of this section is dedicated to the proof of Theorem \ref{opthm}. We briefly outline the underlying intuition. We imagine a fictitious centralized subgradient method which has instant access to all the subgradients $g_i(t)$ obtained at time $t$ and moves towards their average with an appropriately chosen step-size; we show that (i) this subgradient method makes progress towards an optimal solution, despite the fact that the subgradients $g_i(t)$ are all evaluated at different points (ii) the values  ${y}_i(t)$ maintained by each node are not too far from the state of this fictitious subgradient method. Putting (i) and (ii) together with the standard analysis of the subgradient method will imply that the values $y_i(t)$ move towards an optimal solution. More specifically,  the fictitious subgradient method is described by Lemma \ref{avglemma} and the bounds needed in part (ii) are proved in Corollary \ref{distbounds}.

We now start the formal proof of Theorem \ref{opthm}. We begin with some definitions. Our first definition and lemma rewrites the update of Eq. (\ref{optaccel}) in a particularly convenient way. 

\medskip

\begin{definition} Define \[ \q(t) = \left[ \begin{array}{c} \y(t) \\ \z(t) \end{array} \right] \] We will use the notation $B$ to denote the matrix \[  B= 
 \left[ \begin{array}{cc} \alpha M' & - (\alpha-1) M' \\ I  & 0 \end{array} \right] \]  where $\alpha  = 2 - 2/(9U+1)$; note that in the previous section this would have been denoted $B(1)$. Finally, 
we will use $\g(t)$ to denote the vector that stacks up the subgradients $g_i(t), i = 1, \ldots, n$. 
% Finally, given a matrix $A$, we will use $[A]_{1:k, 1:l}$ to denote the principal $k \times l$ submatrix of $A$ (i.e., the matrix %obtained by taking rows $1, \ldots, k$ and columns $1, \ldots, l$). 
\end{definition}

\medskip

\begin{lemma}  We have that
\begin{equation} \q(t+1) = B \q(t) - \beta \left[ \begin{array}{c} \g(t) \\ \g(t) \end{array} \right] \label{delup} \end{equation} 
\end{lemma}

\begin{proof}  Indeed, \begin{eqnarray*} \y(t+1)  =  M' \x(t)  - \beta \g(t)
& = & M' \left( \y(t) + (\alpha - 1) (\y(t) - \z(t) ) \right) - \beta \g(t) \\
& = &  \alpha M' \y(t) - (\alpha-1) M' \z(t) - \beta \g(t)
\end{eqnarray*} while \[ \z(t+1) = \y(t) - \beta \g(t). \] Stacking up the last two equations proves the lemma. 
\end{proof}

\medskip

Thus we see that Eq. (\ref{delup}) is a very compact way of summarizing our protocol. Inspecting this equation, it is natural to suppose that analysis of the properties of multiplication 
by the matrix $B$ will turn out to be essential. Indeed, the following definitions and lemma in fact  describes the action of multiplication by $B$ on a  specific kind of vector in $\R^{2n}$. 

\medskip
 
\begin{definition} We define ${\cal U}(\gamma)$ be the set of vectors $\w \in \R^{2n}$ of the form $\w = \left[ \begin{array}{cc} \w_1 \\ \w_2 \end{array} \right]$ such that $\1^T \w_1 = \1^T \w_2 = \gamma$. 
\end{definition}

\medskip

\begin{definition} For simplicity of notation, we introduce the shorthand $\eta = 1 - \frac{1}{9U}$. 
\end{definition}

\medskip

\begin{lemma} $B$ maps each ${\cal U}(\gamma)$ into itself and, for each $t$, the entries of the three vectors $\x(t), \y(t), \z(t)$ have the same mean. Moreover, if $\uu$ is a vector in $\R^{2n}$ of the form $\uu = \left[ \begin{array}{c} \w \\ \w \end{array} \right]$ then assuming $U \geq n$ we have 
\begin{equation} \left| \left| [B^k \uu]_{1:n} - \frac{\1^T \w}{n} \1  \right| \right|_2^2 \leq 2 \eta^{k} \left| \left| \w  - \frac{\1^T \w}{n} \1 \right| \right|_2^2   \leq 2 \eta^{k} ||\w ||_2^2  . \label{decayb}   \end{equation} \label{intolemma} 
\end{lemma} 

\begin{proof}  To argue that $B$ maps ${\cal U}(\gamma)$ into itself observe that \[ B  \left[ \begin{array}{cc} \w_1 \\ \w_2 \end{array} \right] =  \left[ \begin{array}{cc} \alpha M' \w_1 - (\alpha - 1) M' \w_2 \\ \w_1 \end{array} \right] \] and the conclusion follows from the fact that $\1^T M' = \1^T$.
To argue that the vectors $\x(t), \y(t), \z(t)$ have the same mean observe that this is true at time $t=1$ since we initialize $\y(1)=\z(1)=\x(1)$. For $t>1$, observe that Eq. (\ref{delup}) implies $\y(t),\z(t)$ have the same mean due to the fact that $B$ maps each ${\cal U}(\gamma)$ into itself. Moreover, this further implies that $\x(t)$ has the same mean as $\y(t)$ by  the third line of Eq. (\ref{optaccel}). This concludes the proof that all $\x(t), \y(t), \z(t)$ have the same mean. 

Finally, the first inequality of Eq. (\ref{decayb}) follows from Theorem \ref{constheorem}. Indeed, observe that if the vector $\p(t) = \left[ \begin{array}{c} \y(t) \\ \y(t-1) \end{array} \right]$ is generated via Eq. (\ref{alm}) then $\p(t+1) = B \p(t)$. Recalling that Eq. (\ref{alm}) is preceeded by the initialization $\y(0)=\y(1)$, we may therefore view Theorem \ref{constheorem}  as providing a bound on the squared distance of $(\1^T \w/n) \1$ for the vectors $B^k \left[ \begin{array}{c} \w \\ \w \end{array} \right]$. The second inequality of Eq. (\ref{decayb}) is trivial.  \end{proof}

\medskip

Our next definition introduces notation for the average entry of the vector ${\bf x}(t)$. 

\medskip

\begin{definition} Recall that in the previous section, $\overline{x}$ referred to the average of the vector $\x(0)$, which was also the same as the average of all the vectors $\x(t), \y(t)$.  However, now that $\x(t), \y(t)$ are instead generated from Eq. (\ref{optaccel}) their average will not be constant over time. Consequently, we define 
\[ \overline{x}(t) = \frac{\1^T \x(t)}{n} \] By the previous lemma, $\overline{x}(t)$ is also the average of the entries of $\y(t)$. As before, we will use $\overline{x}$ to denote the {\em initial} average $\overline{x} = \overline{x}(1)$. 
\end{definition} 

\medskip

The following lemma, tracking how the average of the vector $\x(t)$ evolves with time, follows now by inspection of Eq. (\ref{optaccel}). 

\medskip

\begin{lemma} The quantity $\overline{x}(t)$ satisfies the following equation \begin{equation} \label{avgevol} \overline{x}(t+1) = \overline{x}(t) - \beta \frac{\sum_{i=1}^n g_i(t)}{n} \end{equation} \label{avglemma} \end{lemma}

\medskip

With these preliminary observations out of the way, we now turn to the main part of the analysis. We begin with a lemma which bounds how far the vector $\y(t)$ is from the span of the all-ones vector, i.e., bounds which
measure the ``disagreement'' among the $y_i(t)$. 

\medskip

\begin{lemma} Suppose $U \geq n$. Then, \begin{align*} ||\y(t) - \overline{x}(t) \1||_2  & \leq  \beta \left(  \sum_{j=1}^{t-1} \sqrt{2} \eta^{(j-1)/2} ||\g(t-j)||_2 \right) + \sqrt{2} \eta^{(t-1)/2}  ||\y(1) - \overline{x} \1 ||_2 \end{align*} \label{subgrdecay}
\end{lemma}

\begin{proof}  Let us introduce the notation $\hg(t) =  -\left[ \begin{array}{c} \g(t) \\ \g(t) \end{array} \right]$ and $\hh =  \left[ \begin{array}{c} \1 \\ \1 \end{array} \right]$.  As a consequence of Eq. (\ref{delup}) we then have 
\[ \q(t) =  \beta \hg(t-1) + \beta B \hg(t-2) + \cdots + \beta B^{t-2} \hg(1) + B^{t-1} \q(1)  \]
and  using the fact that  $B$ maps each ${\cal U}(\gamma)$ into itself as proved in Lemma \ref{intolemma}, we have 
\[ \overline{x}(t) \hh = \beta \sum_{j=1}^{t-1}  \frac{-\1^T \g(t-j)}{n} \hh +  \overline{x}(1) \hh \] Therefore \begin{align*} \q(t) - \overline{x}(t) \hh & = \beta \sum_{j=1}^{t-1} B^{j-1} \left( \hg(t-j) - \frac{-\1^T \g(t-j)}{n} \hh \right) + B^{t-1} \left(\q(1) - \overline{x}(1) \hh \right), \end{align*} where we used that $B \hh = \hh$. By taking the first $n$ entries of both sides of this
equation and applying Lemma \ref{intolemma}, we obtain the current lemma.  
\end{proof}

\medskip

The next corollary makes the bounds of Lemma \ref{subgrdecay} more explicit by explicitly bounding some of the sums appearing in the lemma. 

\medskip

\begin{corollary} Suppose $U \geq n$ and suppose $|g_i(t)| \leq L$ for all $i=1, \ldots, n$ and $t=1, 2, \ldots$. We then have that  \begin{align} \label{firstcor} ||\y(t) - \overline{x}(t) \1||_1 & \leq  \frac{1}{1-\sqrt{\eta}} \sqrt{2} \beta L n  + \sqrt{2n} \eta^{(t-1)/2} ||\y(1) - \overline{x} \1||_2 \end{align} and \begin{small} \begin{align} \label{secondcor}  \sum_{k=1}^t ||\y(k) - \overline{x}(k) \1||_1  & \leq \frac{1}{1-\sqrt{\eta}}  \beta  \sqrt{2} L n t  + \frac{1}{1-\sqrt{\eta}} \sqrt{2n} ||\y(1) - \overline{x} \1||_2  \end{align} \end{small} and \begin{small}
 \begin{align} \label{hatdiff} ||\widehat{\y}(t) - \widehat{\overline{x}}(t) \1||_1 & \leq  \frac{1}{1-\sqrt{\eta}}  \beta  \sqrt{2} L n + \frac{1}{t}  \frac{1}{1-\sqrt{\eta}} \sqrt{2n} ||\y(1) - \overline{x} \1||_2    \end{align} \end{small}
\label{distbounds}
\end{corollary}

\begin{proof}  First, since,  $||\y(t) - \overline{x}(t) \1||_1 \leq \sqrt{n} ||\y(t) - \overline{x}(t) \1||_2$, we have that 
\begin{align*} ||\y(t) - \overline{x}(t) \1||_1  & \leq  \left(  \beta \sqrt{n} \sum_{j=1}^{t-1} \sqrt{2} \eta^{(j-1)/2} ||\g(t-j)||_2 \right)  + \sqrt{2n} \eta^{(t-1)/2}  ||\y(1) - \overline{x} \1 ||_2  \end{align*} \ Moreover, since each component of $\g(t-j)$ is at most $L$ in absolute value by assumption, we have that $||\g(t-j)||_2 \leq L \sqrt{n}$. Therefore 
\begin{align} \label{midpoint}  ||\y(t) - \overline{x}(t) \1||_1  & \leq  \left(  \beta L n  \sum_{j=1}^{t-1} \sqrt{2} \eta^{(j-1)/2} \right)  + \sqrt{2n} \eta^{(t-1)/2}  ||\y(1) - \overline{x} \1 ||_2  \end{align}   Eq. (\ref{firstcor}) now follows from using the formula for a geometric sum.  Now replacing $t$ by $k$ in Eq. (\ref{firstcor}) and summing up $k$ from $1$ to $t$, we immediately obtain Eq. (\ref{secondcor}). Finally, using 
\begin{eqnarray*} ||\widehat{\y}(t) - \widehat{\overline{x}}(t) \1||_1 & =  & \left| \left| \frac{\y(1) + \cdots + \y(t)}{t} - \frac{\overline{x}(1)  + \cdots + \overline{x}(t) }{t} \1 \right| \right|_1 
 \leq  \frac{1}{t} \sum_{k=1}^t ||\y(k) - \overline{x}(k) \1 ||_1 
\end{eqnarray*}  we obtain Eq. (\ref{hatdiff}). 
\end{proof}

\medskip

%The follwoing corollary follows immediately from Lemma \ref{subgrdecay}. 
%
%\begin{corollary} There is a function ${\rm Tr}\left(n, L, \x(0) \right)$ satisfying \[ {\rm Tr}\left(n, L, \x(0) \right) =  O \left(n %\log (Ln +  ||\x(0) - \overline{x}(0) \1||_1) \right) \] such that if $t = k + {\rm Tr}(n,L, \x(0))$ then \[ |x_i(t) - x_j(t)| \leq  %e^{-k/(12n)} \] 
%\end{corollary} 

The following lemma is a key point in our analysis: we bound the error of our protocol after $T$ steps in terms of the some standard terms appearing in the usual analysis of the subgradient method as well as the ``disagreement'' among entries of the vectors $\y(t)$ and the averaged vector $\widehat{\y}(T)$. 

\medskip

\begin{lemma} Suppose all the subgradients of all $f_i(\cdot)$ are upper bounded by $L$ in absolute value. Then, for any $w \in \R$, we have that 
\begin{eqnarray*} \left( \frac{1}{n} \sum_{i=1}^n f_i(\widehat{y}_i(T)) \right) - f(w) &  \leq &  \frac{ (\overline{x}(1) - w)^2}{2 \beta T}  + \frac{ \beta  L^2  }{2}   +  \frac{2 L}{T n} \sum_{j=1}^T ||\overline{x}(j) \1 - \y(j)||_1 + \frac{L}{n}  ||\widehat{\y}(T) - \widehat{\overline{x}}(T) \1||_1 
\end{eqnarray*} \label{avgperf}
\end{lemma}

\begin{proof} Indeed, by Eq. (\ref{avgevol}) we have that 
\begin{eqnarray*} \left( \overline{x}(t+1) - w \right)^2  & = &   (\overline{x}(t)  - w)^2 + \frac{\beta^2}{n^2} \left( \sum_{i=1}^n g_i(t) \right)^2 - 2  \frac{\beta}{n} \left( \sum_{i=1}^n g_i(t) \right) (\overline{x}(t) - w) 
\end{eqnarray*}
We also have that 
\begin{eqnarray*} g_i(t)( \overline{x}(t) - w) & = & g_i(t) (y_i(t) - w)  + g_i(t) ( \overline{x}(t) - y_i(t) ) \\ 
& \geq & f_i(y_i(t)) - f_i(w) - L |\overline{x}(t) - y_i(t)|  \\
& = & f_i(\overline{x}(t)) - f_i(w) + f_i(y_i(t))  - f_i (\overline{x}(t)) - L |\overline{x}(t) - y_i(t)| \\ 
& \geq & f_i(\overline{x}(t)) - f_i(w) - 2 L |\overline{x}(t) - y_i(t)| 
\end{eqnarray*} where the final inequality used that the subgradients of $f_i$ are bounded by $L$ in absolute value.

Therefore, 
\begin{eqnarray*} ( \overline{x}(t+1) - w)^2 & \leq &  (\overline{x}(t)  - w)^2 + \frac{\beta^2}{n^2} n^2 L^2  - 2  \beta (f(\overline{x}(t)) - f(w)) +  \frac{4L\beta}{n} \sum_{i=1}^n | \overline{x}(t) - y_i(t) | 
\end{eqnarray*}  which implies
\begin{eqnarray*}  2  \beta (f(\overline{x}(t)) - f(w))  \leq  (\overline{x}(t)  - w)^2  -  ( \overline{x}(t+1) - w)^2  + \beta^2 L^2   +  \frac{4L\beta}{n} ||\overline{x}(t) \1 - \y(t)||_1 \end{eqnarray*}
and summing up the left-hand side, 
\[ 2 \beta \sum_{j=1}^{T}  f(\overline{x}(j)) - f(w) \leq (\overline{x}(1) - w)^2 + \beta^2 L^2 T  +  \frac{4L \beta}{n} \sum_{j=1}^T ||\overline{x}(j) \1 - \y(j)||_1   \] which in turn implies 
\[ \frac{1}{T} \sum_{j=1}^{T}  f(\overline{x}(j)) - f(w) \leq \frac{ (\overline{x}(1) - w)^2}{2 \beta T}  + \frac{ \beta  L^2  }{2}  +  \frac{2 L}{T n} \sum_{j=1}^T ||\overline{x}(j) \1 - \y(j)||_1   \]  Thus
\begin{equation} \label{finaleq1} f(\widehat{\overline{x}}(T)) - f(w) \leq  \frac{ (\overline{x}(1) - w)^2}{2 \beta T}  + \frac{ \beta  L^2  }{2}  +  \frac{2 L}{T n} \sum_{j=1}^T ||\overline{x}(j) \1 - \y(j)||_1.  \end{equation} Now using once again the fact that the subgradients of each $f_i(\theta)$ are upper  bounded by $L$ in absolute value, 
\begin{eqnarray*} \left( \frac{1}{n} \sum_{i=1}^n f_i(\widehat{y}_i(T)) \right) - f(w)  & = & \frac{1}{n}  \left( \sum_{i=1}^n f_i(\widehat{y}_i(T)) - f_i(\widehat{\overline{x}}(T)) \right)  + \frac{1}{n} \left(  \sum_{i=1}^n f_i(\widehat{\overline{x}}(T)) \right) - f(w)\nonumber  \\
& \leq & \frac{L}{n} ||\widehat{\y}(T) - \widehat{\overline{x}}(T) \1||_1 + f(\widehat{\overline{x}}(T)) - f(w) \label{finaleq2}
\end{eqnarray*} Now putting together Eq. (\ref{finaleq1}) and Eq. (\ref{finaleq2}) proves the lemma.  
\end{proof} 

\medskip

With all the previous lemmas established, it now remains to put all the pieces together and complete the analysis. Indeed Lemma \ref{avgperf}, just established, bounds the error of our protocol in terms of the disagreements among entries of  the vector $\y(t), t = 1, \ldots, T$ and $\widehat{\y}(T)$. These quantities, in turn, can be bounded using the previous Corollary \ref{distbounds}. Putting these inequalities together and 
choosing correctly the step-size $\beta$ will lead to the proof of the theorem. 

\medskip

\begin{proof}[Proof of Theorem \ref{opthm}] Indeed, plugging the bounds of Corollary \ref{distbounds} into Lemma \ref{avgperf} we obtain 

\begin{eqnarray*} \left( \frac{1}{n} \sum_{i=1}^n f_i(\widehat{y}_i(T)) \right) - f(w) & \leq & \frac{ (\overline{x}(1) - w)^2}{2 \beta T}  + \frac{ \beta  L^2  }{2}   +  \frac{2 L}{T n} \left(  \frac{1}{1-\sqrt{\eta}}  \beta  \sqrt{2} L n T  + \frac{1}{1-\sqrt{\eta}} \sqrt{2n} ||\y(1) - \overline{x} \1||_2 \right) \\ &&  + \frac{L}{n}  \left(\frac{1}{1-\sqrt{\eta}}  \beta  \sqrt{2} L n  + \frac{1}{T}  \frac{ \sqrt{2n} }{1-\sqrt{\eta}}||\y(1) - \overline{x} \1||_2   \right)
\end{eqnarray*}  We may simplify this as 

\begin{eqnarray*} \left( \frac{1}{n} \sum_{i=1}^n f_i(\widehat{y}_i(T)) \right) - f(w) & \leq & \frac{ (\overline{x}(1) - w)^2}{2 \beta T}  + \frac{ \beta  L^2  }{2}  +  \frac{2 \sqrt{2} L^2 \beta}{1 - \sqrt{\eta}}   + \frac{2 \sqrt{2} L}{T \sqrt{n} (1-\sqrt{\eta})}||\y(1) - \overline{x} \1||_2 \nonumber \\ &&  +  \frac{\sqrt{2} L^2 \beta}{1-\sqrt{\eta}}   +  \frac{L \sqrt{2}}{T \sqrt{n} (1-\sqrt{\eta})} ||\y(1) - \overline{x} \1||_2  \label{istep}
\end{eqnarray*} 
 Now since $\sqrt{1 - \frac{1}{9U}} \leq 1 - \frac{1}{18U}$ we have that the $1/(1-\sqrt{\eta})$ terms in the above equation may be upper bounded as
\[ \frac{1}{1-\sqrt{\eta}} =  \frac{1}{1 - \sqrt{1 - 1/(9U)}} \leq \frac{1}{1 - (1 - 1/(18U))} = 18 U \] 
Furthermore, choose any $w^* \in {\cal W}^*$ and $\beta = 1/(L \sqrt{U T})$ so that $$\beta L^2 = \frac{L}{\sqrt{UT}}, ~~~~~~\frac{\beta L^2}{1 - \sqrt{\eta}} \leq \frac{18 L \sqrt{U}}{\sqrt{T}}, ~~~~~~\beta T = \frac{\sqrt{T}}{L \sqrt{U}},$$ Plugging all this in we obtain  
\begin{eqnarray*} \left( \frac{1}{n} \sum_{i=1}^n f_i(\widehat{y}_i(T)) \right) - f(w^*) & \leq  & \frac{ L \sqrt{U} (\overline{x}(1) - w^*)^2}{2 \sqrt{T}} + \frac{  L  }{2 \sqrt{T} \sqrt{U}}  + 36 \sqrt{2} L \sqrt{\frac{U}{T}} + \frac{36 \sqrt{2} L U}{T \sqrt{n} } ||\y(1) - \overline{x} \1||_2 \nonumber  \\ &&+ \frac{18 \sqrt{2} L \sqrt{U}}{\sqrt{T}}  +  \frac{18 \sqrt{2} L U }{\sqrt{n} T} ||\y(1) - \overline{x} \1||_2  \label{exacterr} \end{eqnarray*} 
This proves the first equation of Theorem \ref{opthm}. To prove the second equation of the theorem, let us plug in the our bound for $1/(1-\sqrt{\eta})$ into Eq. (\ref{hatdiff}): \[
||\widehat{\y}(t) - \widehat{\overline{x}}(t) \1||_1 \leq 18 U  \beta  \sqrt{2} L n  + \frac{1}{t} \left( 18 U \sqrt{2n} ||\y(1) - \overline{x} \1||_2  \right) \] Now plugging in $t=T, \beta = 1/(L \sqrt{UT})$ we obtain 
\[ ||\widehat{\y}(T) - \widehat{\overline{x}}(T) \1||_1 \leq \frac{18 \sqrt{2} n \sqrt{U}}{\sqrt{T}} + \frac{ 18 U \sqrt{2n} }{T}  ||\y(1) - \overline{x} \1||_2 \] which implies \begin{small}
\begin{equation} \label{exactdisp} \frac{1}{n} ||\widehat{\y}(T) - \widehat{\overline{x}}(T) \1 ||_1 \leq  18 \sqrt{2} \frac{\sqrt{U}}{\sqrt{T}} + \frac{18 \sqrt{2} U }{\sqrt{n} T} ||\y(1) - \overline{x} \1||_2 \end{equation} \end{small} and now observing that the median of a set of points minimizes the sum of the $l^1$ distances to those points, this implies the second equation of the theorem. \end{proof}

%\subsection{Fixing the feasibility problem \label{feasfix}}
%
%As we mentioned in the introduction, it may be desirable for the nodes to compute a single number $\widehat{\theta}$ satisfying 
%\[ \left(  \frac{1}{n} \sum_{j=1}^n f_i( \widehat{\theta}) \right)  - f(\theta^*)  =  O \left( L \sqrt{\frac{n}{T}} \right)  \]
%This may be easily accomplished using the following observation. 
%
%\begin{lemma} Suppose node $i$ has the value $w_i$ stored in memory. Then it is possible for all  the nodes to figure out which
%$w_i$ results in the smallest $f(\w_i)$ 
%
%\end{proof}

\noindent {\bf Remark III.A:} One might wish for a stronger bound on the error than our quantity  ${\rm Err}(\theta_1, \ldots, \theta_n)$; 
for example, one might wish to show that every node $i$ holds some estimate $\theta_i$ which satisfies
$f(\theta_i) - f(w^*) = O(L \sqrt{n/T})$. 

We remark here that is not hard to modify our protocol to satisfy this condition. Indeed, simply observe that the quantity
$\widehat{\overline{x}}(t)$ satisfies this bound; indeed, our proof proceeds by first bounding 
$f(\widehat{\overline{x}}(t)) - f(w^*)$ and (see Eq. (\ref{finaleq1})) and then arguing that this implies the bound on
${\rm Err}(\widehat{y}_1(T), \ldots, \widehat{y}_n(T))$ by a perturbation argument. Since $\widehat{\overline{T}}$ is the 
average of the values of $\widehat{y}_1(T), \ldots, \widehat{y}_n(T)$, there is a natural two-step procedure the nodes
can implement:  first they can use the protocol of Eq. (\ref{optaccel}) to compute $\widehat{y}_1(T), \ldots, \widehat{y}_n(T)$, and then they 
can average these values using Eq. (\ref{alm}). 

In order for this to result in an $O(L \sqrt{n/T})$ bound on the distance from optimality at each node, each node needs  to obtain an additive $O(\sqrt{n/T})$ approximation to $\widehat{\overline{x}}(t)$. Since by Eq. (\ref{exactdisp}) each node is at most $n^{1.5}/T^{1/2}$ from $\widehat{\overline{x}}(t)$ after the completion of the first stage of the above protocol, this means 
that the averaging protocol of Eq. (\ref{alm}) needs to run for an
$O(n \log n)$  steps during the second stage. Thus this protocol achieves $f(\theta_i) - f(w^*) = O(L \sqrt{n/T})$ for estimates $\theta_i$ at each node
after $T+O(n \log n)$ iterations. 

\bigskip

\noindent {\bf Remark III.B:} We note that there is an alternative and very simple approach to distributed optimization. After initializing their states $\theta_i(1)$ at some identical initial state $\theta_i(1) = \theta$, the nodes 
 pick an error tolerance $\epsilon$ and perform a multi-stage procedure: each node computes the
subgradient $g_i(k)$ of $f_i(\cdot)$ at the the point $\theta_i(k)$; these subgradients are then averaged using an average consensus protocol until each node has the averaged estimate $g_{i}^{\rm ave}(k)$ which differs from the true average  $(1/n) \sum_{i=1}^n g_i(i)$ by at most $\epsilon$; finally, each node updates as $\theta_i(k+1) = \theta_i(k) - \frac{1}{\sqrt{k}} g_i^{\rm ave}(k)$. This two-stage procedure of computing subgradients and approximately averaging them is then repeated for $K$ 
iterations. 

It is not too hard to see that this can be made to work; for example, as long as nodes know an upper bound $U$ on the total number of nodes, they can use the average consensus protocol of Eq. (\ref{alm}) to average the subgradients and they can use
Theorem \ref{constheorem} to decide the number of consensus iterations to perform until all $g_i^{\rm ave}(k)$ differ by at most $\epsilon$ from the true average subgradient. {\em Nevertheless, the scheme of Eq. (\ref{optaccel}) has three  distinct advantages relative to 
this two-stage approach. }

First,  the two-stage protocol is considerably less elegant than any approach which avoids the use of inner/outer loops. 

Secondly, the convergence rates obtained by the two-stage scheme are a logarithmic factor worse compared to the convergence rates from Theorem \ref{opthm}. Indeed, observe that choosing $\epsilon \leq 1/K^{1.5}$ ensures that all nodes are always within an $1/\sqrt{K}$ of each other.  This, however, means that the nodes need to average the subgradients for $O(n \log (L\sqrt{n}K^{1.5}))$ iterations using the average consensus protocol of Eq. (\ref{alm}), since initially the vector that stacks up the subgradients has $2$-norm at most $L\sqrt{n}$. A standard analysis of the subgradient method with errors along the lines of Lemma \ref{avgperf} implies that an error $O(L/\sqrt{K})$ is achieved at each node after $O(K n \log (L\sqrt{n}K^{1.5}))$ iterations, which is  worse than the bound of Theorem \ref{opthm} by a logarithmic factor. Unlike the protocol of Eq. (\ref{optaccel}), the protocol we have just sketched does not achieve the optimal decay $O(1/\sqrt{T})$ with the
number of iterations $T$. 

Finally, although we have not made a point of it here, the scheme we have described in this paper will perform well when the underlying graph is well-connected, even if there is no knowledge of the graph. For example, if the underlying graph is an expander with spectral gap bounded away from $1$, we will have that ${\rm Err}(\cdot, \cdot, \ldots, \cdot)$ will decay as $O(1/\sqrt{T})$ with no dependence on $n$ or $U$. By contrast, the two-stage procedure will need to have an inner loop of size at least $O(U \log (1/\epsilon))$ to accurately average the subgradients (since the nodes do not know the underlying graph is an expander), so that final convergence time will still scale at least linearly in $U$.

\subsection{Two-dimensional grids and geometric random graphs\label{sec:opt:geom}} We can utilize Corollary \ref{fasterprop}, proved in the last section, to show that on two-dimensional grids and geometric random graphs, convergence time is essentially proportional to {\em square root} of the number of nodes.

\medskip

\begin{proposition} Consider the protocol 
\begin{eqnarray*} y_i(t+1) & = & x_i(t) +  \frac{1}{2} \sum_{j \in N(i)} \frac{x_j(t) - x_i(t)}{\max(d(i), d(j))} - \beta g_i(t) \nonumber \\ 
z_i(t+1) & = & y_i(t) - \beta g_i(t) \nonumber \\ 
x_i(t+1) & = & y_i(t+1) + \left(  1 - \gamma \right) \left( y_i(t+1) - z_i(t+1) \right) 
\end{eqnarray*} with the initialization $\y_i(1) = \z_i(1) = \x_i(1)$. Under the same assumptions as Theorem \ref{opthm}, there 
is a choice of $\gamma = \Theta \left( \frac{1}{\sqrt{U \log U}} \right)$ such that with step-size $\beta = \frac{1}{L \sqrt{T} (U \log U)^{1/4}}$ we 
have that if $G$ is a two-dimensional grid, 
\begin{eqnarray*}
{\rm Disp}(\widehat{y}_1(T), \ldots, \widehat{y}_n(T)) & = & O \left(  \frac{(n \log n)^{1/4}}{T^{1/2}} \left( \frac{||\y(1) - \overline{x} \1||_2}{T^{1/2}}  \left( \frac{ \log n}{n} \right)^{1/4}  + 1  \right) \right)  \\
 {\rm Err}(\widehat{y}_1(T), \ldots, \widehat{y}_n(T)) & = & O \left( L \frac{(n \log n)^{1/4}}{T^{1/2}} \left( \frac{||\y(1) - \overline{x} \1||_2}{T^{1/2}}  \left( \frac{ \log n}{n} \right)^{1/4} + 1 + (\overline{x} - w^*)^2 \right) \right). 
\end{eqnarray*}  If $G$ is a geometric random graph with $r^2 \geq \frac{8c \log n}{n}$ where $c>1$, then 
the same equations hold with high probability. 
%there exists a constants $C_1,C_2$ such that 
%with high probability we have that for all $t$, 
%\begin{eqnarray}
%{\rm Disp}(\widehat{y}_1(T), \ldots, \widehat{y}_n(T)) & = &  C_1 \frac{n^{1/4} \log^{1/4} n}{T^{1/2}} \left( C_2 + ||\y(1) - \overline{x}(1) \1||_2 + %(\overline{x}(1) - w^*)^2 \right) \label{meddecay} \\
% {\rm Err}(\widehat{y}_1(T), \ldots, \widehat{y}_n(T)) & = & C_1 \frac{L n^{1/4} \log^{1/4} n}{T^{1/2}} \left( C_2 + ||\y(1) - \overline{x}(1) \1||_2 + %(\overline{x}(1) - w^*)^2 \right) \label{meddecay}  \label{funcdecay}
%\end{eqnarray}
\end{proposition}

\medskip

The proof is almost identical to the proof of Theorem \ref{opthm}, but using the improved bound $1/(1 - \sqrt{\eta}) =  O ( \sqrt{n \log n)}$ which follows from Corollary \ref{fasterprop} for grids or geometric random graphs. We omit the details.

\section{Formation maintenance and leader-following\label{sec:multi}}

In this section, we describe the application of Theorem \ref{constheorem} to two closely related problems: formation maintenance from offset measurements and leader-following. In the former, a collection of nodes must maintain a formation, provided that each node can measure the positions of other nodes in its own  coordinate system. The leader-following problem is similar to the average consensus problem in that all nodes need to agree on the same value, but this value is not the initial average but now the value held by a certain leader. 

\subsection{Formation maintenance from offset measurements\label{sec:form}}

The setup we describe here closely parallels \cite{othesis, krick, ren}, which noted the connection between formation control from offsets and consensus. As in the previous sections, there is a fixed connected graph $G=(V,E)$, which we now assume to be bidirectional\footnote{Bidirectional graphs differ from undirected graphs as follows. An undirected graph $G=(V,E)$  has edge set $E$ composed of {\em unordered pairs}; by contrast, a bidirectional graph has edge set $E$ composed of ordered pairs but $(i,j) \in E$ implies $(j,i) \in E$.}. We will associate with each edge $(i,j) \in E$ a vector ${\bf r}_{ij} \in \R^d$ known to both nodes $i$ and $j$. A collection of points $\p_1, \ldots, \p_n$ in $\R^d$ are said to be ``in formation'' if 
for all $(i,j) \in E$ we have that $\p_j - \p_i = {\bf r}_{ij}$. In this case, we will write $\p_1, \ldots, \p_n \in {\cal F}$. 

Intuitively, we may think of ${\bf r}_{ij}$ as the ``offset'' between nodes $i$ and $j$; for a collection of points in formation, the location ${\bf p}_j$ in the coordinate system where ${\bf p}_i$ is at the origin is exactly ${\bf r}_{ij}$. Note that if $\p_1, \ldots, \p_n \in {\cal F}$ then for any vector ${\bf t} \in \R^d$ we will have $\p_1 + {\bf t}, \ldots, \p_n+{\bf t}\in {\cal F}$. In other words, a translate of any collection of points in formation is also in formation. 

%It goes without saying that there may be choices of ${\bf r}_{ij}$ for which there is no collection of points $\p_1, \ldots, \p_n$ belong to the formation; for %example, this will happen if the sum of ${\bf r}_{ij}$ along some cycle in the graph $G$ is not $0$. 

We will say that the formation defined by ${\bf r}_{ij}, (i,j) \in E$ is {\em valid} if there is at least one collection of points $\p_1, \ldots, \p_n$ in formation. It is easy to see that not every 
choice of ${\bf r}_{ij}$ corresponds to a valid formation: for example, we must have that ${\bf r}_{ij} = - {\bf r}_{ji}$. 

We now consider the problem of formation maintenance with offset measurements: a collection of nodes with initial positions $\p_1(0), \ldots, \p_n(0)$ would like to repeatedly update these positions so that $\p_1(t), \ldots, \p_n(t)$ approaches a collection of points which are in the formation. We will assume that each node $i$ knows the offsets $\p_j(t)-\p_i(t)$ for all of its neighbors $j$ at every time step $t$. However, we do not assume the existence of any kind of global coordinate system common to all the nodes.

This is a considerably simplified model of real-world formation control in several respects. First, we are assuming a first-order model wherein nodes can modify their positions directly; a more realistic model would closely model the dynamics of each agent. As the dynamics of different kinds of agents will be quite different (see, for example, studies of the formation control problem for UAVs in \cite{uav},  unicycles in \cite{unicycle1, unicycle2}, and marine craft \cite{ocean}), here we stick with the simplest possible model. Secondly, we ignore a number of other relevant considerations, such as collision avoidance, velocity alignment, connectivity maintenance, and so forth. 

We note that the assumptions we make are different than the most common assumptions in formation control, wherein one often assumes that agents are able to measure converge to formations defined through distances (e.g., \cite{form1, form2}). In that setup, formations are defined not only up to translations but also up to rotations. Our assumptions are  closest to some of the problems considered in \cite{ren, krick, arcak1} which studied more sophisticated variations of this setup. 

%However, the conclusions we obtain are correspondingly stronger; for example, Corollary \ref{formcor} below proves global geometric %convergence with a linear convergence time, which is much stronger than anything which can be said for formations defined by distances. 

%A natural approach is for each node to iterate as \[ \p_i(t+1) = \p_i(t) - \alpha \frac{ \partial}{\partial {\bf p}_i} \sum_{j \in N(i)} ||\p_j(t) - \p_i(t) - {\bf %r}_{ij}||_2^2 \] for some small $\alpha$. It is easy to see that it is possible to implement this in a distributed way. Futhermore, as shown in \cite{othesis} %this iteration is equivalent to the consensus iteration after a change of coordinates and succesfully drives all the nodes to the formation. 
%
Motivated by Theorem \ref{constheorem}, we propose the following method. Each node initializes $\y_i(1) = \p_i(1)$ and updates as \begin{small}
\begin{eqnarray} \y_i(t+1) &  = & \p_i(t) + \frac{1}{2} \sum_{j \in N(i)} \frac{\p_j(t) - \p_i(t) - {\bf r}_{ij}}{\max(d(i), d(j))} \label{formcontiteration} \\ 
 \p_i(t+1) &  =  & \y_i(t+1) + \left( 1 - \frac{2}{9U+1} \right) \left( {\bf y}_i(t+1) - {\bf y}_i(t) \right)  \nonumber
\end{eqnarray} \end{small}

Note that this protocol may be implemented by each agent in a distributed way, with each agent $i$ needing to know only measurements $\p_j(t) - \p_i(t)$ for all of its neighbors $j \in N(i)$. Our performance bound for this protocol is given in the following proposition and corollary.

\medskip

\begin{proposition}  Suppose that $G=(V,E)$ is a bidirectional, connected graph and  the formation defined by ${\bf r}_{ij}, (i,j) \in E$ is valid. Then given a collection of points $\p_1, \ldots, \p_n$ not in the formation, there is a unique collection of points
$\overline{\p}_1, \ldots, \overline{\p}_n$ such that:
\begin{itemize} \item $\overline{\p}_1, \ldots, \overline{\p}_n$ is in the formation. 
\item \[ \frac{1}{n} \sum_{i=1}^n \overline{\p}_i = \frac{1}{n} \sum_{i=1}^n {\p_i} \] 
\end{itemize} \label{formexist}
\end{proposition}

\medskip

\begin{corollary} Suppose that (i) $G=(V,E)$ is a bidirectional connected graph (ii) the formation defined by ${\bf r}_{ij}, (i,j) \in E$ is valid (iii) $U \geq n$ (iv) every node in an undirected connected graph implements the update of 
Eq. (\ref{formcontiteration}). We then have that \[ \sum_{i=1}^n ||\y_i(t) - \overline{\p}_i||_2^2 \leq 2 \left( 1 - \frac{1}{9U} \right)^{t-1} \sum_{i=1}^n ||\y_i(1) - \overline{\p}_i||_2^2 \] \label{formcor}
\end{corollary}

\medskip

Summarizing,  the protocol of Eq. (\ref{formcontiteration}) drives all $\y_i(t)$ to the point in the formation that has the same center of mass as the initial collection of positions. Furthermore, the convergence is geometric with constant $1-1/(9U)$. Under our assumption that $U$ is within a constant factor of $n$, this gives the linear time bound $O \left( n \ln {\sqrt{\sum_{i=1}^n ||\y_i(1) - \overline{\bf p}_i(1)||_2^2}}/{\epsilon} \right)$ until all nodes are within a distance of $\epsilon$ of their final destinations. 

It remains the prove the above proposition and corollary. 

\medskip

\begin{proof}[Proof of Proposition \ref{formexist}] The existence of such a $\overline{\p}_1, \ldots, \overline{\p}_n$ may be established as follows. 
By our assumption that the formation is valid, let ${\q}_1, \ldots, {\q}_n$ be a collection of points in formation; setting 
\[ \overline{\p}_i = \q_i + \frac{1}{n} \sum_{j=1}^n ( \p_j - \q_j ) \] it is immediate that we have obtained $\overline{\p}_1, \ldots, \overline{\p}_n$
with the desired properties. 

The uniqueness of $\overline{\p}_1, \ldots, \overline{\p}_n$ is an immediate consequence of the fact that two collections of points in formation must be translates of each other. Indeed, if $\q_1, \ldots, \q_n \in {\cal F}$ and ${\cal P}_{ij}$ is a path from $i$ to $j$ in $G$, then
\[ \q_j = \q_i + \sum_{j \in {\cal P}_{ij}} {\bf r}_{ij} \] 
It follows that if $\z_1, \ldots, \z_n$ is another collection of points in the formation, then $\z_j - \q_j = \z_i - \q_i$ for all $i,j=1,\ldots,n$, i.e., $\z_1, \ldots, \z_n$ is a translate of $\q_1, \ldots, \q_n$. 
\end{proof}

\begin{proof}[Proof of Corollary \ref{formcor}] Define
\begin{eqnarray*} \uu_i(t) & = & \p_i(t) - \overline{\p}_i \\ 
\w_i(t) & =  & \y_i(t) - \overline{\p}_i
\end{eqnarray*} We can subtract the identity 
\[ \overline{\p}_i = \overline{\p}_i + \frac{1}{2} \sum_{j \in N(i)} \frac{\overline{\p}_j - \overline{\p}_i - {\bf r}_{ij}}{\max(d(i), d(j))} \]
from the first line of Eq. (\ref{formcontiteration}) to obtain 
\begin{equation} \label{topform} \w_i(t+1) = \uu_i(t) + \frac{1}{2} \sum_{j \in N(i)} \frac{\uu_j(t) - \uu_i(t)}{\max(d(i), d(j))} 
\end{equation} Furthermore, subtracting $\overline{\p}_i$ from the second line of Eq. (\ref{formcontiteration}), we obtain \begin{small}
\begin{equation} \label{bottomform} \uu_i(t+1) = \w_i(t+1) + \left( 1 - \frac{2}{9U+1} \right) (\w_i(t+1) - \w_i(t)) 
\end{equation} \end{small}

Considering Eq. (\ref{topform}) and Eq. (\ref{bottomform}) together, we see that the vectors obtained by stacking up the $i$'th components of 
$\uu_i(t), \w_i(t)$ satisfy Eq. (\ref{alm}). Furthermore, \begin{small} \[ \frac{1}{n} \sum_{i=1}^n \w_i(1) = \frac{1}{n} \sum_{i=1}^n \y_i(1) - \overline{\p}_i = \frac{1}{n} \sum_{i=1}^n \p_i(1) - \frac{1}{n} \sum_{i=1}^n \overline{\p}_i = {\bf 0},\] \end{small} and therefore for each $j=1, \ldots, d$, the initial average of the $j$'th entries of the vectors $\w_i(1)$ is $0$. Consequently,
applying Theorem \ref{constheorem} we obtain that for all $j=1, \ldots, d$, 
\[ \sum_{i=1}^n [\w_i(t)]_j^2 \leq 2 \left( 1 - \frac{1}{9U} \right)^{t-1} \sum_{i=1}^n [\w_i(1)]_j^2 \] From the definition of $\w_i(t)$, we therefore have \begin{small}
\[ \sum_{i=1}^n \left( [\y_i(t)]_j - [\overline{\p}_i]_j \right)^2 \leq 2 \left( 1 - \frac{1}{9U} \right)^{t-1} \sum_{i=1}^n \left( [\y_i(1)]_j - [\overline{\p}_i]_j \right)^2  \] 
\end{small} Summing up over $j=1, \ldots, d$ yields the statement of the corollary.

\end{proof}

\subsection{Leader-following} 

The motivation for studying leader-following comes from the problem of velocity alignment:  a single node keeps its velocity fixed to 
some $\vv \in \R^d$, and all other nodes repeatedly update their velocities based only on observations of neighbors' velocities; the goal is for the velocity of each agent to converge to the leader's velocity $\vv$. A slightly more general setup, which we will describe here, involves a nonempty subset of leaders $S \subset \{1, \ldots, n \}$ all of whom fix their velocities to the same $\vv$. 

The leader-following problems is naturally motivated by the problems of controlling multi-agent systems. For example, it is desirable for a flock of autonomous vehicles to be controlled by a single (possibly human) agent which steers one of the vehicles and expects the rest to automatically follow. A number of variations of this problem have been considered, following
its introduction in \cite{morse}; we refer the reader to \cite{ leader3, leader5, leader4} for some more recent references. 

Here our goal is to design a protocol for leader-following with linear convergence time based on Theorem \ref{constheorem}. Specifically, we propose that, while each leader node keeps $\x_i(t) = \vv$, each non-leader node $i$ initializes $\y_i(1) = \x_i(1)$ and updates as 
\begin{eqnarray} \y_i(t+1) & = & \x_i(t) + \frac{1}{2} \sum_{j \in N_i(t)} \frac{\x_j(t) - \x_i(t)}{\max(d(i), d(j))} \label{lfu} \\ 
\x_i(t+1) & = & \y_i(t+1) + \left( 1 - \frac{2}{18U+1} \right) ( \y_i(t+1) - \y_i(t) ) \nonumber
\end{eqnarray}

Our performance bound for this protocol is given in the following corollary.

\medskip

\begin{corollary} Suppose (i) $G$ is a connected, undirected graph (ii) each node $i$ in a nonempty set $S$ sets $\x_i(t) = \vv$ for all $t$
(ii) each node not belonging to $S$ updates using Eq. (\ref{lfu}) (iv) $U \geq n$. Then
\[ \sum_{i=1}^n ||\y_i(t) - \vv||_2^2 \leq 2 \left( 1 - \frac{1}{18U} \right)^{t-1} \sum_{i=1}^n ||\y_i(1) - \vv||_2^2 \] \label{lfucor}
\end{corollary} 

\medskip

 Note that, as before, under the assumption that $U$ is within a constant factor of $n$, this corollary gives  a linear $O \left( n \left( \ln {\sqrt{\sum_{i=1}^n ||\y_i(1) - \vv||_2^2}}/{\epsilon} \right) \right)$ bound on the number of updates until each node is within a distance of $\epsilon$ of $\vv$. 

\medskip

\begin{proof}[Proof of Corollary \ref{lfucor}] Without loss of generality, let us make the following three  assumptions. First, let us relabel the nodes so that the first $k$ nodes comprise all the non-leader nodes, i.e.,  $S = \{1, \ldots, k\}^c$. Secondly, by observing that the Eq. (\ref{lfu}) can be decoupled along the components of the vectors $\x_i(t), \y_i(t)$, we can without loss of generality assume these vectors belong to $\R^1$, i.e., they are numbers. Thus $v$ will be a number as well, and in the rest of this proof, $\x(t),\y(t)$ will denote the vectors whose $i$'th components are the numbers $x_i(t), y_i(t)$. Finally, since adding the same  constant to each entry of $x_i(1), y_i(1)$, as well as to $v$, adds the same constant to all $x_i(t), y_i(t)$, we can assume without loss of generality that $v=0$.

We now proceed to the proof itself, which is an application of a trick from \cite{julien} which we now sketch. We will show that the dynamics of Eq. (\ref{lfu}) are 
nothing more than the dynamics of Eq. (\ref{alm}) on a modified undirected graph $G'$ which has at most twice as many vertices. The graph $G'$ can be obtained by splitting each non-leader agent $i \in S^c$ into two distinct nodes, one with the same initial value as $i$, and the other with a {\em negated} initial value. Informally speaking, this will allow us to claim that all the leader agents perform the consensus iterations of Eq. (\ref{alm}), {\em which will keep their value at $0$ by symmetry}, rather than keeping their value fixed at $v=0$ by design. We can then apply Theorem \ref{constheorem} to bound the convergence times of the dynamics of Eq. (\ref{alm}) on the new graph with extra states. This idea originates from \cite{julien} where it is used to analyze consensus protocols with negative weights. 

Formally, we construct a new undirected graph $G' = (V',E')$ as follows. For each node $i \in S^c$ (recall that by assumption the set of leader agents $S$ satisfies $S^c = \{1, \ldots, k\}$) we create  two nodes, $i^A, i^B$. Thus $V' = \{1^A, \ldots, k^A, 1^B, \ldots, k^B, k+1, \ldots, n\}$. Next, for every edge $(i,j) \in E$ where both $i,j \in S^c$, we put two edges $(i^A,j^A), (i^B, j^B)$ into $E'$; and for every edge $(i,j) \in E$ where one of $i,j$ in $S$ and the other is not - say $i \in S, j \in S^c$ - we put the two edges $(i, j^A), (i, j^B)$ into $E'$. 

Note that because $G$ is connected by assumption, we have that $G'$ is connected; indeed, for any pair of nodes $i,j$ in  $G'$, there exists a path from $i$ to $j$ by putting together two paths from their corresponding nodes in $G$ to an arbitrarily chosen leader node. 

Finally, we set initial values $x_i'(1) = 0$ for all leader nodes $i \in S$; and for non-leader nodes $i \in S^c$, we set $x_{i^A}'(1) = x_i(1)$, $x_{i^B}'(1) = - x_i(1)$. As before, 
we initialize $y_i(1) = x_i(1)$.

The following statement now follows immediately by induction: suppose we run Eq. (\ref{alm}) on the undirected graph $G'$ with initial conditions $\x'(1), \y'(1)$ and the upper bound $2U$ on the number of nodes and obtain the vectors $\x'(t), \y'(t)$; and further suppose we run Eq. (\ref{lfu}) on the graph $G$ with initial conditions $\x(1), \y(1)$ to obtain the vectors $\x(t), \y(t)$;  then for all $i = 1, \ldots, k$ and all iterations $t$, 
$x_i(t) = x_{i^A}(t)$ and $y_i(t) = y_{i^A}(t)$.

Applying Theorem \ref{constheorem} we obtain that 
\[ ||\y'(t)||_2^2 \leq 2 \left( 1 - \frac{1}{18U} \right)^{t-1} ||\y'(1)||_2^2. \] Noting that, for all $t$, $||\y'(t)||_2^2 = 2 ||\y(t)||_2^2$, we have that this equation implies the current corollary. \end{proof} 

\section{Simulations: average consensus and distributed median computation\label{sec:simul}} We now describe several simulations designed to show the performance of our protocol on some particular graphs. 

We begin with some simulations of our protocol for average consensus. In Figure \ref{cons-simul}, we simulate the  protocol of Eq. (\ref{alm}) on the line graph and the lollipop graph\footnote{The lollipop graph is composed of a complete graph on $n/2$ nodes connected to a line of $n/2$ nodes.}. These graphs are natural examples within the context of consensus protocols. The line graph represents a long string of nodes in tandem, while the lollipop graph and its close variations have typically been used as examples of graphs on which consensus performs poorly; see, for example, \cite{land-odl}. The initial vector was chosen to be  $x_1(0)=1,  x_i(0)=0$ in both cases, and we assume that $U=n$, i.e., each node knows the total amount of nodes in the system. In both cases, we plot the number of nodes on the x-axis and the first time $||\x(t) - \overline{x} \1||_{\infty} < 1/100$ on the y-axis. 

The plots show the protocol of Eq. (\ref{alm}) exhibiting a curious, somewhat uneven behavior on these examples. Nevertheless, the linear bound of Theorem \ref{constheorem} appears to be consistent with the fairly quick convergence to consensus shown in the figure. 

\begin{figure*} 
\begin{tabular}{cc} 
\includegraphics[scale=0.45]{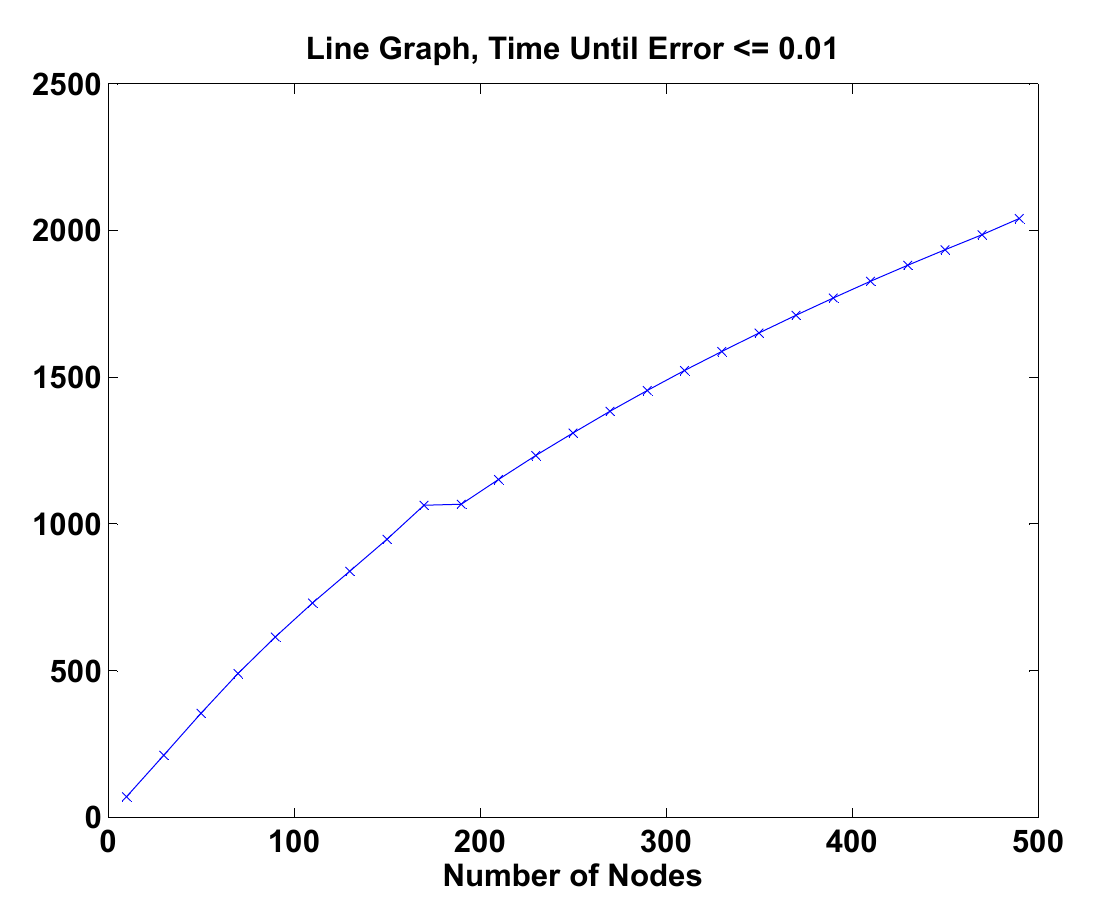} 
&  \includegraphics[scale=0.45]{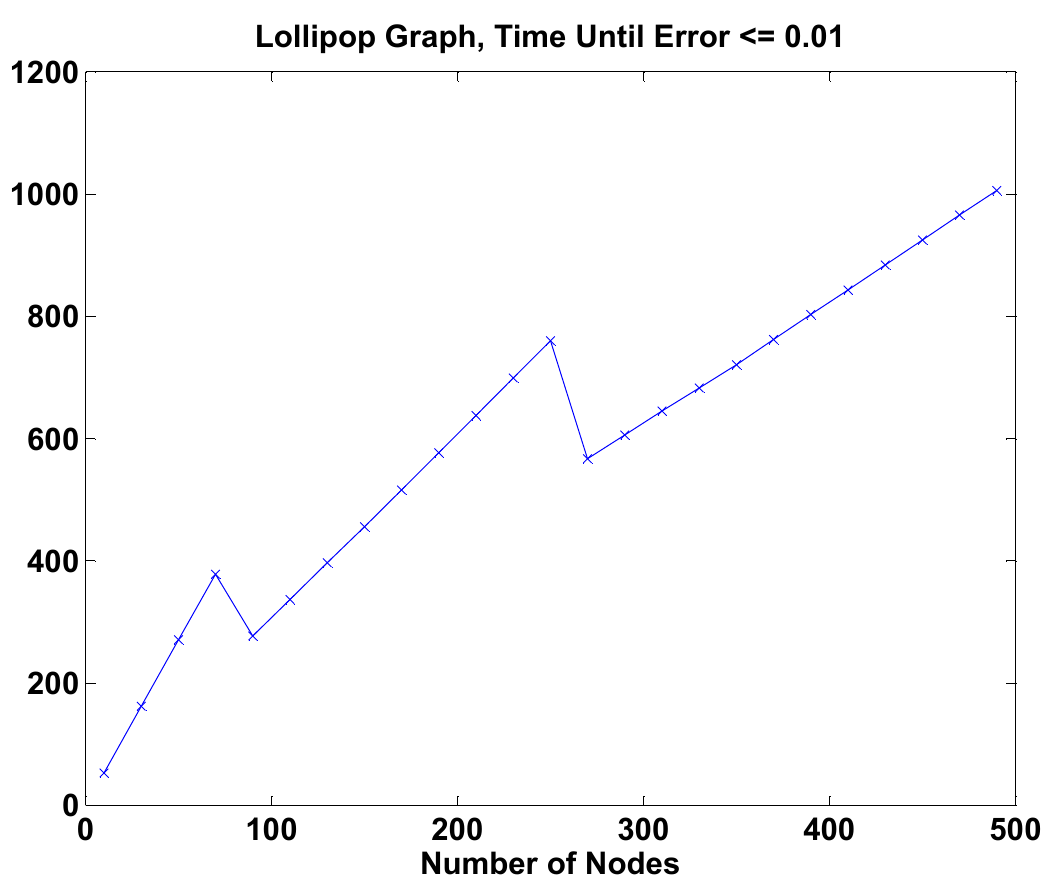} 
\end{tabular} \caption{Convergence time as a function of the number of nodes for the average consensus protocol of Eq. (\ref{alm}). Convergence time for the line graph is shown on the left and for the lollipop graph is shown on the right.} \label{cons-simul}
\end{figure*}

\begin{figure*} 
\begin{tabular}{cc} 
\includegraphics[scale=0.45]{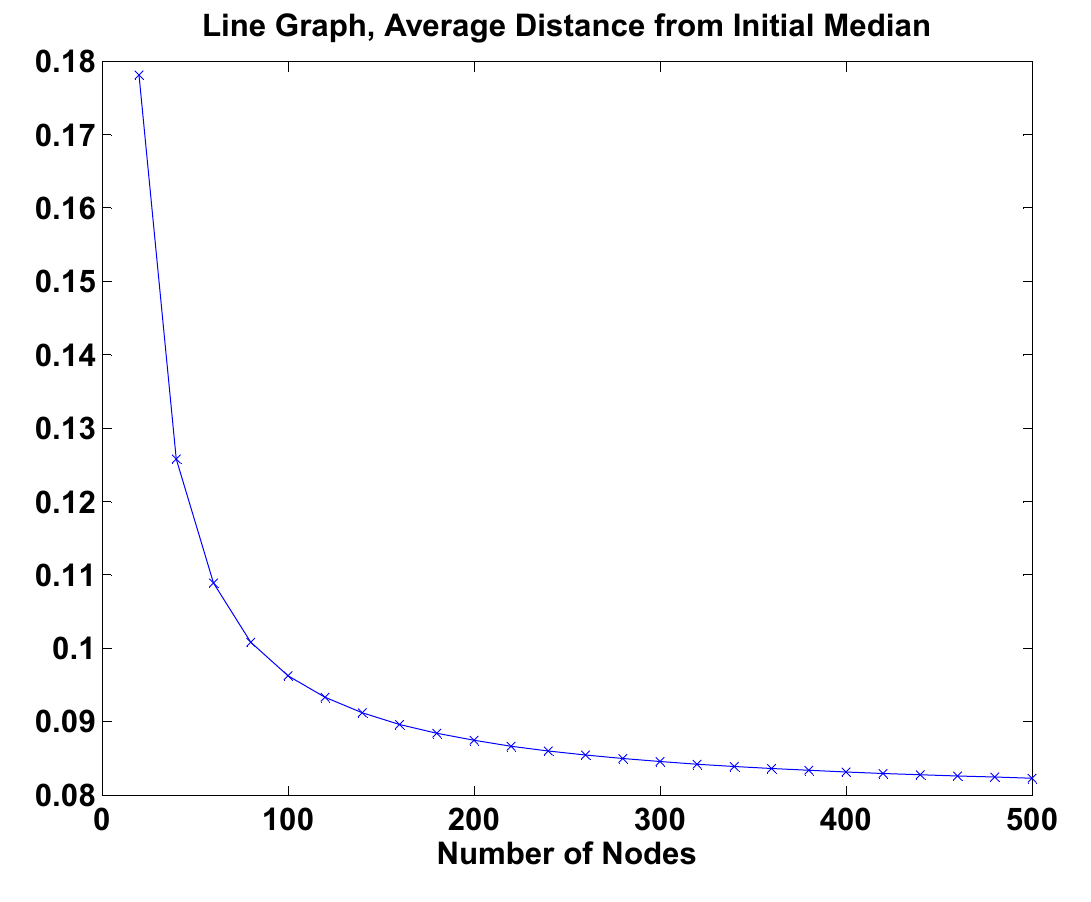} 
&  \includegraphics[scale=0.45]{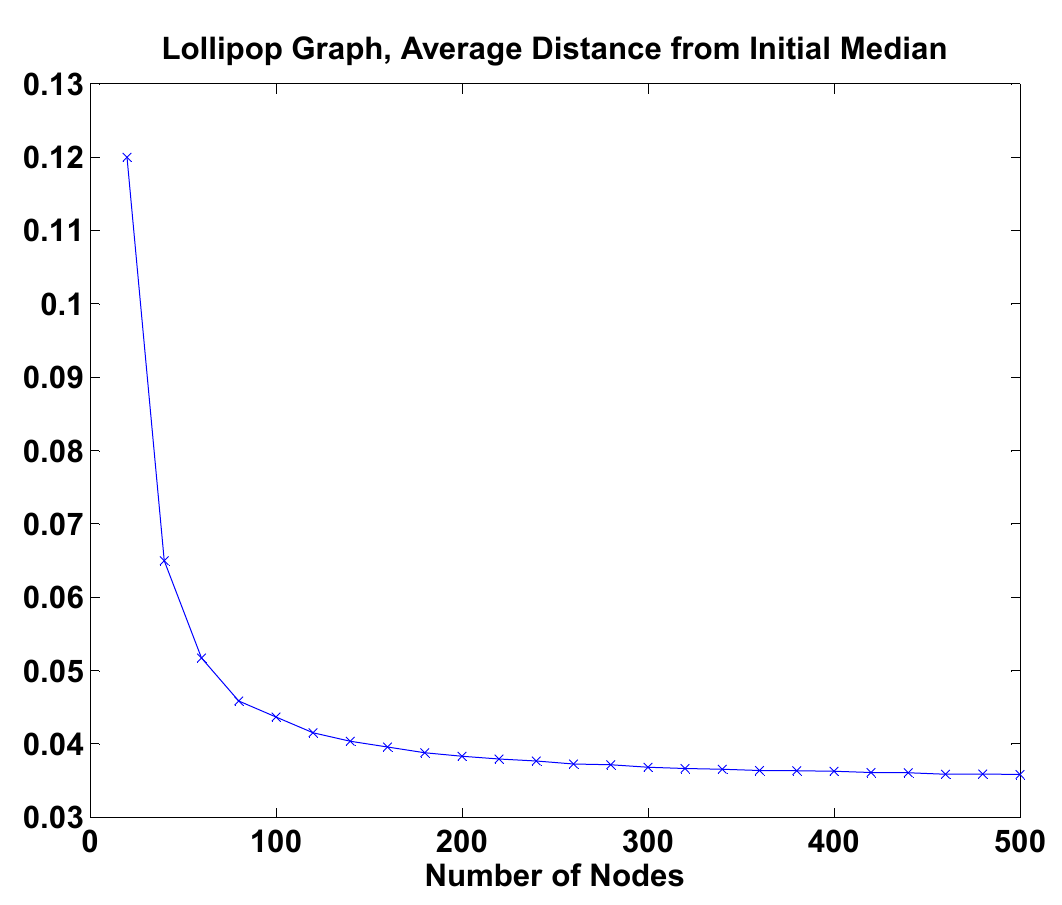} 
\end{tabular} \caption{Average deviation from the median as a function of the number of nodes using the decentralized optimization protocol of Eq. (\ref{optaccel}) after $T=4n$ iterations. The line graph is shown on the left and and convergence time the lollipop graph is shown on the right.} \label{opt-simul}
\end{figure*}

We next consider the problem of median computation: each node in a network starts with a certain value $w_i$ and the nodes would like to 
agree on a median of these values. This can be cast within the decentralized optimization framework by observing that the median of numbers $w_1, \ldots, w_n$ is a solution of the minimization problem
\[ \arg \min_{\theta} ~ \sum_{i=1}^n |\theta-w_i| \] and therefore can be computed in a distributed way using the update of Eq. (\ref{optaccel}) with the convex functions $f_i(\theta)=|\theta-w_i|$. One may take $L=1$, since the subgradients of the absolute value function are bounded by one in magnitude.

We implement the protocol of Eq. \ref{optaccel} on both the line graph and lollipop graph. In both cases, we once again take $U=n$, i.e., we assume that every node knows the total number of nodes in the system. We suppose each node $o$ starts with $x_i(0)=w_i$. Moreover, for $i=1, \ldots, n/2$, we set $w_i$ to be the remainder of $i$ divided by $10$, and
$w_{n/2+i} = -w_i$.  This choice of starting values is motivated by three considerations: (i) our desire to have a simple answer (in this case, the median is always zero) (ii) the desire to have an initial condition whose norm does not grow quickly with $n$ (since our convergence times for decentralized optimization scale with the initial condition as well and we want to focus on the effect of network size)
(iii) the obseervation that this choice creates a ``bottleneck'' - on the line graph, for example, numbering the nodes from $1, \ldots, n$ going from left to right, this results in the left half of the nodes having a positive initial condition and the right half of the nodes having negative initial condition; it is interesting to observe whether our
protocol will be able to cope with such bottlenecks quickly.

Theorem \ref{opthm} suggests good performance will be attained when the number of iterations $T$ is linear in $n$.  Correspondingly, we set $T=4n$. In
Figure \ref{opt-simul}, we plot the number of nodes on the x-axis vs $(1/n) ||\widehat{\y}(T)||_1$ (which is the average deviation from a correct answer since $0$ is a median) on the y-axis.  As can be seen from the figure, choosing a linear number of iterations $T=4n$ clearly suffices to compute the median of integers in the range $[-10,10]$ with excellent accuracy.  

\section{Conclusion\label{sec:concl}} We have presented a protocol for the average consensus problem on a fixed, undirected graph with a linear convergence time in the number of nodes in the network. Furthermore, we have shown one can use modifications of this protocol to obtain linear convergence times in decentralized optimization, formation control, and leader-following.  Given the copious literature on applications of consensus  throughout multi-agent control, we expect that our results may be used to derive linearly convergent protocols for many other problems.

The most natural open problem following this work is to extend the results of this paper to time-varying graphs. The mobility inherent in most multi-agent systems means that communication graphs will vary, often unpredictably so. As of writing, the best convergence time attainable on time-varying sequences of undirected graphs is quadratic in the number of nodes \cite{four}; for time-varying directed graphs, no convergence time which is even polynomial in the number of nodes appears to be known.

\end{document}